\documentclass[12pt,reqno]{amsart}

\usepackage{amssymb}

\usepackage[all]{xy}
\usepackage{graphicx,amssymb,amsfonts,epsfig,amsthm,a4,amsmath,url}
\usepackage{latexsym, mathrsfs, mathtools, dsfont}
\usepackage[latin1]{inputenc}
\usepackage{enumerate}
\usepackage{upgreek}
\usepackage{tikz-cd}
\usepackage[colorlinks=true,linkcolor=blue,citecolor=blue]{hyperref}

\newcounter{ictr}

\newcounter{actr}

\newtheorem{thm}{Theorem}[section]

\newtheorem{lem}[thm]{Lemma}

\newtheorem{prop}[thm]{Proposition}

\newtheorem*{corr}{Corollary}

\theoremstyle{definition}

\newtheorem{defn}[thm]{Definition}

\newtheorem{conj}[thm]{Conjecture}

\newtheorem{defns}[thm]{Definition}

\theoremstyle{remark}

\newtheorem{open}[thm]{Open Question}

\numberwithin{equation}{section}

\newcommand{\hhs}{{Hilbert-Hadamard space}}
\newcommand{\fin}{\text{fin}}

\newcommand{\GL}{\text{GL}}

\newcommand{\dist}{\mathrm{dist}}
\newcommand{\tr}{\mathrm{tr}}

\title{Higher invariants in noncommutative geometry 
}

\dedicatory{Dedicated to Alain Connes with great admiration}

\author{Zhizhang Xie}
\address[Zhizhang Xie]{Department of Mathematics,  Texas A\&M University, College Station, TX 77843, USA}
\email{xie@math.tamu.edu}
\thanks{The first author is partially supported by NSF 1500823, NSF 1800737.}

\author{Guoliang Yu}
\address[Guoliang Yu]{Department of Mathematics, Texas A\&M University, College Station, TX 77843, USA
\\    Shanghai Center for Mathematical Sciences, Shanghai, China}
\email{guoliangyu@math.tamu.edu}
\thanks{The second author is partially supported by NSF 1700021, NSF 1564398}

\thanks{}

\begin{document}

\begin{abstract}
 We give a survey on  higher invariants in noncommutative geometry and their applications to differential geometry and topology. 
  \end{abstract}

\maketitle


\section{Introduction}

 Geometry and topology of a smooth manifold is often governed by natural differential operators on the manifold. 
 When a smooth manifold $M$ is closed (compact without boundary), 
 a basic invariant of these differential operators is  their Fredholm index.
Roughly speaking the Fredholm index measures the size of the solution space 
for an infinite dimensional linear system associated to the operator $D$. More precisely, 
the Fredholm index  of $D$ by the formula:
$\text{index}(D)=\text{dim} (\text{kernel}(D)) -\text{dim} (\text{kernel} (D^\ast)). $  
 The beauty of the Fredholm index is its invariance under small perturbations and homotopy equivalence. The Fredholm index is an obstruction to invertibility of the operator. The Fredholm index of such an operator $D$
 is computed by the well-known Atiyah-Singer index theorem   
  \cite{AS}. 
 The Atiyah-Singer index theorem has important applications to geometry,  topology, and mathematical physics.

Alain Connes' powerful noncommutative geometry provides the framework for a much more refined  index theory, called higher index theory \cite{BC, BCH, C, CM, K}. Higher index theory is a far-reaching generalization of classic Fredholm index theory by taking into consideration of the symmetries given by the fundamental group.
Let $D$ be an elliptic differential operator on a closed manifold $M$ of dimension $n$. If $\widetilde M$ is the universal cover of $M$, and  $\widetilde D$ is the lift of $D$ onto $\widetilde M$, then we can define a higher index of $\widetilde D$  in $K_n(C^\ast_r(\pi_1 M))$, where $\pi_1 M$ is the fundamental group of $M$ and $K_n(C^\ast_r(\pi_1 M))$ is the $K$-theory of the reduced group $C^\ast$-algebra $C^\ast_r(\pi_1 M)$. This higher index  is an obstruction to the invertibility of $\widetilde D$ and is invariant under homotopy. Higher index theory plays a fundamental role in the studies of  problems in geometry and topology such as the Novikov conjecture on homotopy invariance of higher signatures and the Gromov-Lawson conjecture on nonexistence of Riemmanian metrics with positive scalar curvature. 
Higher indices are often referred to as primary invariants due to its homotopy invariance property.  The Baum-Connes conjecture provides an algorithm for computing the higher index \cite{BC, BCH} while the strong Novikov conjecture predicts when the higher index vanishes \cite{K}.
When a closed manifold $M$ carries a Riemannian metric with positive scalar curvature, by the Lichnerowicz formula, the Dirac operator  $\widetilde D$ on $\widetilde M$ is invertible and hence its higher index vanishes. If $M$ is aspherical, i.e. its universal cover is contractible, then the strong Novikov conjecture predicts that the higher index of the Dirac operator is non-zero and hence implies the Gromov-Lawson conjecture stating that any closed aspherical manifold cannot carry a Riemannian metric with positive scalar curvature \cite{R}. 
Another important application of higher index theory is
the Novikov conjecture \cite{N}, a central problem in topology.
Roughly speaking,
the Novikov conjecture claims that compact smooth manifolds are rigid at an infinitesimal level. More precisely, 
the Novikov conjecture states that the higher signatures of compact oriented smooth manifolds are invariant under orientation preserving homotopy equivalences.  Recall that a compact manifold is called aspherical if its universal cover is contractible.
In the case of aspherical manifolds, the Novikov conjecture is an infinitesimal version of the Borel conjecture, which states that all compact aspherical manifolds
are topologically rigid, i.e. if another compact manifold $N$ is homotopy equivalent to the given compact aspherical manifold $M$, then $N$ is homeomorphic to $M$.  A theorem of Novikov states that the rational Pontryagin classes are invariant under orientation preserving homeomorphisms  
 \cite{N1}. Thus the Novikov conjecture for compact aspherical manifolds follows from the Borel conjecture and Novikov's theorem,  since for aspherical manfolds, the information about higher signatures is equivalent to that of rational Pontryagin classes. In general, the Novikov conjecture follows from the strong Novikov conjecture when applied to the signature operator.  With the help of noncommutative geometry, 
  spectacular progress has been made on the Novikov conjecture.

When the higher index  of an elliptic operator is trivial with a given trivialization, a secondary index theoretic invariant naturally arises \cite{HR2, Roe1}.   This secondary invariant is  called the higher rho invariant.  It serves as an obstruction to  locality of the inverse of an invertible elliptic operator. For example,  when  a closed manifold $M$ carries a Riemannian metric with positive scalar curvature,  the Dirac operator  $\widetilde D$  on its universal cover $\widetilde M$ is invertible, hence its higher index is trivial. In this case,  the positive scalar curvature metric gives a specific trivialization of the higher index, thus naturally defines a higher rho invariant. Such a secondary index theoretic invariant is of fundamental importance for studying the space of positive scalar curvature metrics of a given closed spin manifold. For instance,  this secondary invariant is an essential ingredient for measuring the size of the moduli space (under diffeomorphism group) of positive scalar curvature metrics on a given closed spin manifold \cite{XY1}.	
The following is another typical situation where higher rho invariants naturally arise.  Given   an orientation-preserving homotopy equivalence between two oriented closed manifolds,  the higher index of the signature operator on the disjoint union of the two manifolds (one of them equipped with the opposite orientation) is trivial with a trivialization given by the homotopy equivalence. Hence such a homotopy equivalence naturally defines a higher rho invariant for the signature operator \cite{HR2, Roe1}. More generally, the notion of higher rho invariants can be defined for homotopy equivalences between topological manifolds \cite{Z}, and these invariants serve as a powerful tool for  detecting whether a homotopy equivalence can be ``deformed" into a homeomorphism.  Furthermore, the authors proved  in \cite{WXY} that the higher rho invariant defines a group homomorphism on the structure group of a \emph{topological} manifold. As an application, one can use the higher rho invariant to measure the degree of nonrigidity of a topological manifold.

Connes' cyclic cohomology theory provides a powerful method to compute  higher rho invariants. It turns out that   the pairing of cyclic cohomology with
higher rho invariants can be computed in terms of John Lott's higher eta invariants. This relation can be used to give an elegant approach to the higher Atiyah-Patodi-Singer index theory for manifolds with boundary and provide a potential way to construct counter examples to the Baum-Connes conjecture.

The purpose of this article is to give a friendly survey on these recent developments of higher invariants in noncommutative geometry and their applications to geometry and topology.

\subsection*{Acknowledgment}
The authors  wish to thank  Alain Connes for numerous inspiring discussions. 

\section{Geometric $C^\ast$-algebras}\label{sec:prem}

In this section, we give an overview of several $C^\ast$-algebras  naturally arising from geometry and topology. The K-theory groups of these  $C^\ast$-algebras serve as receptacles of our higher invariants. 

Let $X$ be a proper metric space. That is, every closed ball in $X$ is compact. An $X$-module is a separable Hilbert space equipped with a	$\ast$-representation of $C_0(X)$, the algebra of all continuous functions on $X$ which vanish at infinity. An	$X$-module is called nondegenerate if the $\ast$-representation of $C_0(X)$ is nondegenerate. An $X$-module is said to be standard if no nonzero function in $C_0(X)$ acts as a compact operator. 

We shall first recall the concepts of propagation, local compactness, and pseudo-locality.

\begin{defn}
	Let $H_X$ be a $X$-module and $T$ a bounded linear operator acting on $H_X$. 
	\begin{enumerate}[(i)]
	
		\item The propagation of $T$ is defined to be 
		$sup\{ d(x, y)\mid (x, y)\in supp(T)\}$, where $supp(T)$ is  the complement (in $X\times X$) of the set of points $(x, y)\in X\times X$ for which there exist $f, g\in C_0(X)$ such that $gTf= 0$ and $f(x)\neq 0$, $g(y) \neq 0$;
		\item $T$ is said to be locally compact if $fT$ and $Tf$ are compact for all $f\in C_0(X)$; 
		\item $T$ is said to be pseudo-local if $[T, f]$ is compact for all $f\in C_0(X)$.  
	\end{enumerate}
\end{defn}

Pseudo-locality is the essential property for the concept of an abstract ``differential operator" in K-homology theory \cite{A, K}.

The following concept  was introduced by Roe in his work on higher index theory for noncompact spaces \cite{Roe}. 

\begin{defn}\label{def:localg}
	Let $H_X$ be a standard nondegenerate $X$-module and $\mathcal B(H_X)$ the set of all bounded linear operators on $H_X$. 
	The Roe algebra of $X$, denoted by $C^\ast(X)$, is the $C^\ast$-algebra generated by all locally compact operators with finite propagations in $\mathcal B(H_X)$.
	
	\end{defn}
	
	The following localization algebra was introduced by \cite{Y}. 
	
	\begin{defn}\label{def:localg}
	
	The localization algebra $C_L^\ast(X)$   is the $C^\ast$-algebra generated by all bounded and uniformly norm-continuous functions $f: [0, \infty) \to C^\ast(X)$   such that 
		\[ \textup{propagation of $f(t) \to 0 $, as $t\to \infty$. }\]  
	We define $C_{L, 0}^\ast(X)$ to be  the kernel of the evaluation map 
		\[  e \colon C_L^\ast(X) \to C^\ast(X),  \quad   e(f) = f(0).\]
		In particular, $C_{L, 0}^\ast(X)$ is an  ideal of $C_L^\ast(X)$.

\end{defn}

The localization algebra was motivated by local index theory.

Now we take symmetries into consideration. Let's  assume that a discrete group $\Gamma$ acts properly on  $X$  by isometries. Let $H_X$ be a $X$-module equipped with a covariant unitary representation of $\Gamma$. If we denote the representation of $C_0(X)$ by $\varphi$ and the representation of $\Gamma$ by $\pi$, this means 
\[  \pi(\gamma) (\varphi(f) v )  =  \varphi(f^\gamma) (\pi(\gamma) v),\] 
where $f\in C_0(X)$, $\gamma\in \Gamma$, $v\in H_X$ and $f^\gamma (x) = f (\gamma^{-1} x)$. In this case, we call $(H_X, \Gamma, \varphi)$ a covariant system.  

\begin{defns}[\cite{Y3}]
	A covariant system $(H_X, \Gamma, \varphi)$ is called admissible if 
	\begin{enumerate}[(1)]
		\item the $\Gamma$-action on $X$ is proper and cocompact;
		\item $H_X$ is a nondegenerate standard $X$-module;
		\item for each $x\in X$, the stabilizer group $\Gamma_x$ acts on $H_X$ regularly in the sense that the action is isomorphic to the action of $\Gamma_x$ on $l^2(\Gamma_x)\otimes H$ for some infinite dimensional Hilbert space $H$. Here $\Gamma_x$ acts on $l^2(\Gamma_x)$ by translations and acts on $H$ trivially. 
	\end{enumerate}
\end{defns}
We remark that for each locally compact metric space $X$ with a proper and cocompact isometric action of $\Gamma$, there exists an admissible covariant system $(H_X, \Gamma, \varphi)$. Also, we point out that the condition $(3)$ above is automatically satisfied if $\Gamma$ acts freely on $X$. If no confusion arises, we will denote an admissible covariant system $(H_X, \Gamma, \varphi)$ by $H_X$ and call it an admissible $(X, \Gamma)$-module.

\begin{defn}
	Let $X$ be a locally compact metric space $X$ with a proper and cocompact isometric action of $\Gamma$. If $H_X$ is an admissible $(X, \Gamma)$-module, we denote by $\mathbb C[X]^\Gamma$ the \mbox{$\ast$-algebra} of all $\Gamma$-invariant locally compact operators with finite propagations in $\mathcal B(H_{X})$.  We define the equivariant Roe algebra $C^\ast(X)^\Gamma$ to be the completion of $\mathbb C[X]^\Gamma$ in $\mathcal B(H_{ X})$.
\end{defn}

We remark that if the $\Gamma$-action on $X$ is cocompact, then 
the equivariant Roe algebra $C^\ast(X)^\Gamma$ is $\ast$-isomorphic to
$C^\ast_r(\Gamma)\otimes K$, where $C^\ast_r(\Gamma)$ is the reduced group $C^\ast$-algebra of $\Gamma$ and $K$ is the $C^\ast$-algebra of all compact operators. We also point out that, up to isomorphism, $C^\ast(X) = C^\ast(X, H_X)$ does not depend on the choice of the standard nondegenerate $X$-module $H_X$. The same statement holds for  $C_L^\ast(X)$,  $C_{L,0}^\ast( X)$ and  their $\Gamma$-equivariant versions. 

We can also define the maximal versions of the geometric $C^\ast$-algebras in this section by taking the norm completions over all $\ast$-representations of their algebraic counterparts.

\section{Higher index theory and localization }\label{sec:index}

In this section, we construct the higher index of an elliptic operator.
We also introduce  a local index map from the $K$-homology group to the $K$-group of  the localization algebra and explain that this local index map is an isomorphism. 

\subsection{K-homology}
	We first  discuss  the $K$-homology theory of  Kasparov.
	Let $X$ be a locally compact metric space with a proper and cocompact
	isometric action of $\Gamma$. The $K$-homology groups $K_{j}^{\Gamma
	}(X)$, $j=0, 1$, are generated by the following cycles modulo certain
	equivalence relations (cf. \cite{K}):
	\begin{enumerate}[(ii)]
		\item[(i)] an even cycle for $K_{0}^{\Gamma}(X)$ is a pair $(H_{X},
		F)$, where $H_{X}$ is an admissible $(X, \Gamma)$-module and $F\in
		\mathcal{B}(H_{X})$ such that $F$ is $\Gamma$-equivariant, $F^{\ast} F
		- I$ and $F F^{\ast} - I$ are locally compact and $[F, f] = F f - fF $
		is compact for all $f\in C_{0}(X)$.
		\item[(ii)] an odd cycle for $K_{1}^{\Gamma}(X)$ is a pair $(H_{X},
		F)$, where $H_{X}$ is an admissible $(X, \Gamma)$-module and $F$ is a
		$\Gamma$-equivariant self-adjoint operator in $\mathcal{B}(H_{X})$
		such that $F^{2} - I$ is locally compact and $[F, f]$ is compact for
		all $f\in C_{0}(X)$.
	\end{enumerate}
	
Roughly speaking, the $K$-homology group of $X$ is generated by abstract
elliptic operators over $X$ \cite{A, K}.
	
		In the general case where the action of $\Gamma$ on $X$ is not
		necessarily cocompact, we define
		\begin{eqnarray*} K_{i}^{\Gamma}(X) = \varinjlim_{Y\subseteq X}
			K_{i}^{\Gamma}(Y),
		\end{eqnarray*}
		where $Y$ runs through all closed $\Gamma$-invariant subsets of $X$
		such that $Y/\Gamma$ is compact.

	\subsection{K-theory and boundary maps}\label{sec:kt}
	In this subsection, we recall the standard construction of the index
	maps in $K$-theory of $C^{\ast}$-algebras.
	For a short exact sequence of $C^{\ast}$-algebras $ 0 \to\mathcal{J}
	\to\mathcal{A} \to\mathcal{A}/\mathcal{J} \to0$,
	we have a six-term exact sequence in $K$-theory:
	\begin{eqnarray*}
		\xymatrix{ K_0( \mathcal J ) \ar[r] & K_0(\mathcal A) \ar[r] &
			K_0(\mathcal A/\mathcal J ) \ar[d]^{\partial_1} \\
			K_1(\mathcal A/\mathcal J) \ar[u]^{\partial_0}
			& K_1(\mathcal A) \ar[l] & K_1(\mathcal J) \ar[l]
		}
	\end{eqnarray*}
	Let us recall the definition of the boundary maps $\partial_{i}$.
	\begin{enumerate}[(2)]
		\item[(1)] Even case. Let $u$ be an invertible element in $\mathcal
		{A}/\mathcal{J}$. Let $v$ be the inverse of $u$ in $\mathcal{A}/\mathcal
		{J}$. Now suppose $U, V\in\mathcal{A}$ are lifts of $u$ and $v$. We define
		\begin{eqnarray*} W =
			\begin{pmatrix} 1 & U\\ 0 & 1
			\end{pmatrix}
			\begin{pmatrix} 1 & 0 \\ -V & 1
			\end{pmatrix}
			\begin{pmatrix} 1 & U \\ 0 & 1
			\end{pmatrix}
			\begin{pmatrix} 0 & -1\\ 1 & 0
			\end{pmatrix}
			.
		\end{eqnarray*}
		Notice that $W$ is invertible and a direct computation shows that
		\begin{eqnarray*} W -
			\begin{pmatrix} U & 0 \\ 0 & V
			\end{pmatrix}
			\in\mathcal{J}.
		\end{eqnarray*}
		Consider the idempotent
		%
		\begin{equation}\label{eq:keven}
		P = W
		\begin{pmatrix} 1 & 0 \\ 0 & 0
		\end{pmatrix}
		W^{-1} =
		\begin{pmatrix} UV + UV(1-UV) & (2 - UV)(1-UV) U \\ V(1-UV) & (1-VU)^{2}
		\end{pmatrix}
		.
		\end{equation}
		We have
		\begin{eqnarray*} P -
			\begin{pmatrix} 1 & 0 \\0 & 0
			\end{pmatrix}
			\in\mathcal{J}.
		\end{eqnarray*}
		By definition,
		\begin{eqnarray*} \partial([u]) := [P] - \left[
			\begin{pmatrix} 1 & 0 \\0 & 0
			\end{pmatrix}
			\right] \in K_{0}(\mathcal{J}).
		\end{eqnarray*}
		\item[(2)] Odd case. Let $q $ be an idempotent in $\mathcal{A}/\mathcal
		{J}$ and $Q$ a lift of $q$ in $\mathcal{A}$. Then
		\begin{eqnarray*} \partial([q]) := [e^{2\pi iQ}] \in K_{1}(\mathcal
			{J}).
		\end{eqnarray*}
	\end{enumerate}
	
	\subsection{Higher index map and local index map}\label{sec:localind}
	In this subsection, we  describe the constructions of the higher index map
	 \cite{BC, BCH, K, FM}
	and the local index map \cite{Y, Y1}. 
	
	Let $(H_{X}, F)$ be an even cycle for $K_{0}^{\Gamma}(X)$. Choose a
	$\Gamma$-invariant locally finite open cover $\{U_{i}\}$ of $X$ with
	diameter $(U_{i}) < c$ for some fixed $c > 0$. Let $\{\phi_{i}\} $ be a
	$\Gamma$-invariant continuous partition of unity subordinate to $\{
	U_{i}\}$. We define
	\begin{eqnarray*} \mathcal{F} = \sum_{i} \phi^{1/2}_{i} F \phi^{1/2}_{i},
	\end{eqnarray*}
	where the sum converges in strong operator norm topology. It is not
	difficult to see that $(H_{X}, \mathcal{F})$ is equivalent to $(H_{X},
	F)$ in $K^{\Gamma}_{0}(X)$. By using the fact that $\mathcal{F}$ has
	finite propagation, we see that $\mathcal{F}$ is a multiplier of
	$C^{\ast}(X)^{\Gamma}$ and,  is a unitary modulo $C^{\ast
	}(X)^{\Gamma}$. Consider the short exact sequence of $C^{\ast}$-algebras
	\begin{eqnarray*} 0 \to C^{\ast}(X)^{\Gamma} \to\mathcal{M}(C^{\ast
		}(X)^{\Gamma}) \to\mathcal{M}(C^{\ast}(X)^{\Gamma}) /C^{\ast
	}(X)^{\Gamma} \to0
\end{eqnarray*}
where $\mathcal{M}(C^{\ast}(X)^{\Gamma})$ is the multiplier algebra
of $C^{\ast}(X)^{\Gamma}$.
By the construction in Section~\ref{sec:kt} above, $\mathcal{F}$
produces a class $\partial ([\mathcal{F}] )\in K_{0}(C^{\ast}(X)^{\Gamma})$. We
define the higher index of $(H_{X}, F)$ to be $\partial ([\mathcal{F}])$. From now
on, we denote $[\mathcal{F}]$ by $\text{\textup{Ind}}(H_{X}, F)$ or simply
$\text{\textup{Ind}}(F)$, if no confusion arises.

To define the local index  of $(H_{X}, F)$, we need to use a
family of partitions of unity. More precisely, for each $n\in\mathbb
{N}$, let $\{U_{n, j}\}$ be a $\Gamma$-invariant locally finite open
cover of $X$ with diameter $(U_{n,j}) < 1/n$ and $\{\phi_{n, j}\}$ be
a $\Gamma$-invariant continuous partition of unity subordinate to $\{
U_{n, j}\}$. We define
%
\begin{equation}\label{eq:path}
\mathcal{F}(t) = \sum_{j} (1 - (t-n)) \phi_{n, j}^{1/2} \mathcal{F}
\phi_{n,j}^{1/2} + (t-n) \phi_{n+1, j}^{1/2} \mathcal{F} \phi_{n+1, j}^{1/2}
\end{equation}
for $t\in[n, n+1]$.

Then $\mathcal{F}(t), 0\leq t <\infty$, is a multiplier of $C^{\ast
}_{L}(X)^{\Gamma}$ and a unitary modulo $C^{\ast}_{L}(X)^{\Gamma}$.
 By the construction in Section~\ref{sec:kt} above, we  define  $\partial ([\mathcal{F}(t)])\in
K_{0}(C^{\ast}_{L}(X)^{\Gamma})$ to be the local index of $(H_{X}, F)$. If
no confusion arises, we denote this local index class by $\text{\textup
	{Ind}}_{L}(H_{X}, F)$ or simply $\text{\textup{Ind}}_{L}(F)$.

Now let $(H_{X}, F)$ be an odd cycle in $K_{1}^{\Gamma}(X)$. With the
same notation from above, we set $q = \frac{\mathcal{F} + 1}{2}$. Then
the index class of $(H_{X}, F)$ is defined to be $ [e^{2\pi i q}]\in
K_{1}(C^{\ast}(X)^{\Gamma})$. For the local index class of $(H_{X},
F)$, we use $q(t) = \frac{\mathcal{F}(t) + 1}{2}$ in place of $q$.

We have the following commutative diagram:

\[
\begin{tikzcd}[row sep=2.5em]
  &  K_{\ast}^{\Gamma}(X) \arrow{dr}{ Ind} \arrow[dl, "Ind_L"']  \\
K_{\ast}(C^{\ast}_{L}(X)^{\Gamma})  \arrow{rr}{e_\ast} && K_{\ast}(C^{\ast}(X)^{\Gamma}),
\end{tikzcd}
\]
where $e_\ast$ is the homomorphism induced by the evaluation map $e$ at $0$.

The following result was proved in the case of simplicial complexes 
in \cite{Y} and the general case  in \cite{QR}.

\begin{thm} If a discrete group $\Gamma$ acts properly on a locally compact space $X$, then
the local index map is an isomorphism from the $K$-homology group
$K_{\ast}^{\Gamma}(X)$ to the K-group of the localization algebra
$K_{\ast}(C^{\ast}_{L}(X)^{\Gamma})$.
\end{thm}

\section{The Baum-Connes assembly and a local-global principle }
In this section, we formulate the Baum-Connes conjecture as a local-global principle and discuss its connection to the Novikov conjecture.

We first recall the concept of Rips complexes.

\begin{defn} Let $\Gamma$ be a discrete group, let $F\subseteq \Gamma$ be a finite symmetric subset containing the identity (symmetric in the sense if $g\in F$, then $g^{-1}\in F$). The Rips complex $P_F (\Gamma)$ is a simplicial complex such that

\begin{enumerate}[(i)]
		\item the set of vertices is $\Gamma$;
		
		\item a finite subset $\{\gamma_0, \cdots, \gamma_n\}$ span a simplex if and only if $\gamma_i^{-1}\gamma_j \in F$ for all $0\leq i, j\leq n$.
\end{enumerate}

\end{defn}

We endow the Rips complex with the simplicial metric, i.e. the maximal metric whose restriction to a maximal simplex is the standard Euclidean metric on  the simplex.

The Baum-Connes conjecture \cite{BC, BCH} can be stated as follows. 

\begin{conj}[Baum-Connes Conjecture]The evaluation map $e$ induces an isomorphism $e_\ast$ from
the $K$-group of the equivariant localization algebra 
$\displaystyle \varinjlim_{F} K_{\ast}(C^{\ast}_{L}(P_F(\Gamma))^{\Gamma})$
to the \mbox{$K$-group} 
of the equivariant Roe algebra $\displaystyle  \varinjlim_{F} K_{\ast}(C^{\ast}(P_F(\Gamma))^{\Gamma})$, where the limit is taken over the directed set
of all finite symmetric subset $F$ of $\Gamma$ containing the identity. 

\end{conj}

Note that $\displaystyle  \varinjlim_F K_{\ast}(C^{\ast}(P_F(\Gamma))^{\Gamma})$
is isomorphic to K-group of $C_r^\ast(\Gamma)$,
 the reduced group $C^\ast$-algebra of $\Gamma$
since the $\Gamma$ action on the Rips complex is cocompact.

While the K-theory of  the equivariant Roe algebra is global and hard to compute, the K-theory of the localization algebra is local and completely computable. Thus the Baum-Connes conjecture is a local-global principle. If true, the conjecture provides an algorithm for computing $K$-groups of equivariant Roe algebras and higher indices of elliptic operators. 
In particular, in this case, we see that every element in the  $K$-group of  the equivariant Roe algebra
can be localized.

More generally, if $A$ is a $C^\ast$-algebra with an action of $\Gamma$, then we can define the equivariant Roe algebra with coefficients in $A$, denoted by $C^{\ast}(P_F(\Gamma), A)^{\Gamma}$. The  equivariant Roe algebra with coefficients in $A$ is $\ast$-isomorphic to   $(A\rtimes \Gamma)\otimes \mathcal K $, where $\mathcal K$ is the algebra of compact operators on a Hilbert space. We can similarly introduce an equivariant localization algebra with coefficients to formulate the Baum-Connes conjecture with coefficients.


Higson and Kasparov developed an index theory of certain differential operators on an infinite-dimensional Hilbert space and proved the following spectacular result \cite{HK}.

\begin{thm}\label{HKforBC} If a discrete group $\Gamma$ acts on Hilbert space properly and isomentrically,
then the Baum-Connes conjecture with coefficients holds for $\Gamma.$
\end{thm}

Recall that an isometric action $\alpha$ of a group $\Gamma$ on a Hilbert space $H$ is said to be proper if $ \|  \alpha(\gamma)h \| \rightarrow \infty$ when $\gamma \rightarrow \infty$ for any $h\in H$, i.e. for  any $h\in H$ and  any positive number $R>0$, there exists a finite subset $F$ of $\Gamma$ such that $\| \alpha(\gamma) h \| > R$ if $\gamma \in \Gamma - F$.
A theorem of Bekka-Cherix-Valette states that an amenable group acts properly and isometrically on a Hilbert space \cite{BCV}. Roughly speaking, a group is amenable if there exist large finite subsets of the group with small boundary.  The concept of amenability is a large scale geometric property and  was introduced by von Neumann. We refer the readers to the book \cite{NY} as a general reference for geometric group theory related to the Novikov conjecture.

The following deep theorem is due to Lafforgue \cite{L1}.

\begin{thm} The Baum-Connes conjecture with coefficients holds for hyperbolic groups.
\end{thm}

Earlier Lafforgue developed a Banach KK-theory to attack the Baum-Connes conjecture \cite{L}.  This approach yielded the Baum-Connes conjecture  for hyperbolic groups \cite{L, MY}.

The Baum-Connes conjecture with coefficients actually fail for general groups. 
  Higson-Lafforgue-Skandalis gave a counter-example in \cite{HLS}. On the other hand, the Baum-Connes conjecture (without coefficients) is still open.

\section{The Novikov conjecture}

A central problem in topology is the Novikov conjecture.
Roughly speaking,
the Novikov conjecture claims that compact smooth manifolds are rigid at an infinitesimal level. More precisely, 
the Novikov conjecture states that the higher signatures of compact oriented smooth manifolds are invariant under orientation preserving homotopy equivalences.  Recall that a compact manifold is called aspherical if its universal cover is contractible.
In the case of aspherical manifolds, the Novikov conjecture is an infinitesimal version of the Borel conjecture, which states that all compact aspherical manifolds
are topologically rigid, i.e. if another compact manifold $N$ is homotopy equivalent to the given compact aspherical manifold $M$, then $N$ is homeomorphic to $M$.  A theorem of Novikov says that the rational Pontryagin classes are invariant under orientation preserving homeomorphisms  
 \cite{N1}. Thus the Novikov conjecture for compact aspherical manifolds follows from the Borel conjecture and Novikov's theorem,  since for aspherical manfolds, the information about higher signatures is equivalent to that of rational Pontryagin classes. In general, the Novikov conjecture follows from the (rational) strong  Novikov conjecture. 

The (rational) strong Novikov conjecture can be stated as follows.

\begin{conj}[Strong Novikov Conjecture] The evaluation map $e$ induces an injection $e_\ast$ from
the $K$-group of the equivariant localization algebra 
$\displaystyle  \varinjlim_F K_{\ast}(C^{\ast}_{L}(P_F(\Gamma))^{\Gamma})$
to the $K$-group 
of the equivariant Roe algebra $\displaystyle  \varinjlim_F K_{\ast}(C^{\ast}(P_F(\Gamma))^{\Gamma})$, where the limit is taken over the directed set
of all finite symmetric subset $F$ of $\Gamma$ containing the identity. The rational strong Novikov conjecture states that $e_\ast$ is an injection after tensoring with $\mathbb Q$.
\end{conj} 

The strong Novikov conjecture predicts when the higher index of an elliptic operator is non-zero.
The strong Novikov conjecture implies the following analytic Novikov conjecture.

\begin{conj}[Analytic Novikov Conjecture] The evaluation map $e$ induces an injection $e_\ast$ from
the $K$-group of the equivariant localization algebra 
$\displaystyle  K_{\ast}(C^{\ast}_{L}(E\Gamma)^{\Gamma})$
to the $K$-group 
of the equivariant Roe algebra $\displaystyle   K_{\ast}(C^{\ast}(E\Gamma)^{\Gamma})$, where $K_{\ast}(C^{\ast}_{L}(E\Gamma)^{\Gamma})$
is defined to be 
  $\displaystyle \varinjlim_X  K_{\ast}(C^{\ast}_{L}(X)^{\Gamma})$ with the limit to be
taken over the directed set
of locally compact,  $\Gamma$-equivariant, $\Gamma$-cocompact subset $X$ of the universal space  $E\Gamma$ for free $\Gamma$-action, and similarly 
$\displaystyle   K_{\ast}(C^{\ast}(E\Gamma)^{\Gamma})$ is defined to be the limit
$\displaystyle \varinjlim_X  K_{\ast}(C^{\ast}(X)^{\Gamma})$. The rational analytic Novikov conjecture states that $e_\ast$ is an injection after tensoring with $\mathbb Q$, that is, \[ e_\ast \colon K_{\ast}(C^{\ast}_{L}(E\Gamma)^{\Gamma})\otimes \mathbb Q\to K_{\ast}(C^{\ast}(E\Gamma)^{\Gamma})\otimes \mathbb Q \]
is an injection. 
\end{conj} 



 The classical Novikov conjecture follows from the rational analytic Novikov conjecture. 
 With the help of noncommutative geometry, 
  spectacular progress has been made on the Novikov conjecture. It
has been proven that The Novikov conjecture holds  when the fundamental group of the manifold lies in one of the following classes of groups:

\begin{enumerate}
	\item groups acting properly and isometrically on simply connected and non-positively curved manifolds \cite{K}, 
	\item hyperbolic groups \cite{CM}, 
	\item groups acting properly and isometrically on Hilbert spaces \cite{HK},
	\item groups acting properly and isometrically on bolic spaces \cite{KS}, 
	\item groups with finite asymptotic dimension \cite{Y1},
	\item groups coarsely embeddable into Hilbert spaces \cite{Y2}\cite{H}\cite{STY}, 
	\item groups coarsely embeddable into Banach spaces with property (H) \cite{KY}, 
	\item all linear groups and subgroups of all almost connected Lie groups \cite{GHW}, 
	\item  subgroups of the mapping class groups \cite{Ha}\cite{Ki}, 
	\item subgroups of $\operatorname{Out}(F_n)$, the outer automorphism groups of the free groups \cite{BGH},
	\item groups acting properly and isometrically on (possibly infinite dimensional) admissible  Hilbert-Hadamard spaces,
	in particular geometrically discrete subgroups of the group of volume preserving diffeomorphisms of any smooth compact manifold \cite{GWY}. 
\end{enumerate}
In the first three cases, an isometric action of a discrete group $\Gamma$ on a metric space $X$ is said to be \emph{proper} if for some $x\in X$,  $d(x, gx)\rightarrow \infty$ as $g \rightarrow \infty$,  i.e. for  any $x\in X$ and  any positive number $R>0$, there exists a finite subset $F$ of $\Gamma$ such that $ d(x, gx)> R$ if $g \in \Gamma - F$.

In a tour de force, Connes proved a striking theorem that the Novikov conjecture holds for higher signatures associated to Gelfand-Fuchs classes \cite{C1}. Connes, Gromov, and Moscovici proved the Novikov conjecture for higher signatures associated to Lipschitz group cohomology classes \cite{CGM}. Hanke-Schick and Mathai proved the Novikov conjecture  for higher signatures associated to group cohomology classes with degrees one and two \cite{HS}\cite{Ma}.

J. Rosenberg discovered an important application of the (rational) strong Novikov conjecture to the existence problem of Riemannian metrics with  positive scalar curvature \cite{R}. We refer to Rosenberg's survey \cite{R1} for recent developments on this topic.

\subsection{Non-positively curved groups and hyperbolic groups}
\label{ }

In this subsection, we give a survey on the work of A. Mishchenko, G. Kasparov,
A. Connes and H. Moscovici, G. Kasparov and G. Skandalis on the Novikov conjecture for non-positively curved groups and Gromov's hyperbolic groups.

In  \cite{M}, A. Mishchenko introduced a theory of infinite dimensional Fredholm representations of discrete groups to prove the following theorem.

\begin{thm} The Novikov conjecture holds if the fundamental group of a manifold 
acts properly,  isometrically and cocompactly on a simply connected manifold with non-positive sectional curvature.
\end{thm}

In \cite{K}, G. Kasparov developed  a bivariant K-theory, called KK-theory, to prove the following theorem.

\begin{thm} The Novikov conjecture holds if the fundamental group of a manifold 
acts properly and isometrically  on a simply connected manifold with non-positive sectional curvature.
\end{thm}

As a consequence, G. Kasparov proved the following striking theorem. 

\begin{thm} The Novikov conjecture holds if the fundamental group of a manifold 
is a discrete subgroup of a Lie group with finitely many connected components.
\end{thm}

The theory of hyperbolic groups was developed by Gromov \cite{G}. Gromov's hyperbolic groups are generic among all finitely presented groups. A. Connes and H. Moscovici proved the following spectacular theorem using powerful techniques from noncommutative geometry \cite{CM}.

\begin{thm}\label{thm:cm} The Novikov conjecture holds if the fundamental group of a manifold 
is a hyperbolic group in the sense of Gromov.
\end{thm}

The proof of Theorem $\ref{thm:cm}$ uses Connes' theory of cyclic cohomology in a crucial way. Cyclic homology theory plays the role of de Rham theory in noncommutative geometry, and is the natural receptacle for the Connes-Chern character \cite{C}.

The following theorem of G. Kasparov and G. Skandalis unified the above results
\cite{KS}.

\begin{thm} The Novikov conjecture holds if the fundamental group of a manifold 
is bolic.
\end{thm}

Bolicity is a notion of non-positive curvature. Examples of bolic groups include groups acting properly and isometrically on simply connected manifolds with non-positive sectional curvature and Gromov's hyperbolic groups.

\subsection{Amenable groups, groups with finite asymptotic dimension and coarsely embeddable groups} 
\label{}

In this subsection, we give a survey on the work of Higson-Kasparov 
on the Novikov conjecture for amenable groups, the work of G. Yu on the Novikov conjecture for groups with finite asymptotic dimension, and the work of G. Yu,
N. Higson, Skandalus-Tu-Yu on the Novikov conjecture for groups coarsely embeddable into Hilbert spaces. Finally we discuss the work of Kasparov-Yu on the connection of the Novikov conjecture with Banach space geometry.

As mentioned above (Theorem $\ref{HKforBC}$), Higson and Kasparov proved that the Baum-Connes conjecture holds for groups that act properly and isometrically on a Hilbert space \cite{HK}. As a consequence, the Novikov conjecture holds for these groups. 

\begin{thm} The Novikov conjecture holds if the fundamental group of a manifold 
acts properly  and isometrically on a Hilbert space.
\end{thm}

Since amenable groups act properly and isometrically on a Hilbert space \cite{BCV}, the above theorem has the following immediate corollary. 
\begin{corr}
The Novikov conjecture holds if the fundamental group of a manifold is amenable.
\end{corr}

This corollary is quite striking since the geometry of amenable groups can be very complicated (for example, the Grigorchuk's groups \cite{Gr}).

Next we recall a few basic concepts from geometric group theory.
A non-negative function $l$ on a countable group $G$ is called a length function if 
(1) $l(g^{-1})=l(g)$ for all $g\in G$; (2) $l(gh)\leq l(g) +l(h)$ for all $g$ and $h$ in $G$; (3) $l(g)=0$ if and only if $g=e$, the identity element of $G$.
We can associate a left-invariant length metric $d_l$ to $l$: $d_l(g,h)=l(g^{-1}h)$ for all $g,h\in G$. A length metric is called proper if the length function is a proper map (i.e. the inverse image of every compact set is finite in this case). 
It is not difficult to show that every countable group $G$ has a proper length metric. If $l$ and $l'$ are two proper length functions on $G$, then their associated length metrics are coarsely equivalent.  If $G$ is a  finitely generated group and $S$ is a finite symmetric generating set (symmetric in the sense that if an element is in $S$, then its inverse is also in $S$), then we can define the  word length $l_S$ on $G$ by
$$l_S(g)=\min \{n: g=s_1\cdots s_n, s_i\in S\}.$$
If $S$ and $S'$ are two finite symmetric generating sets of $G$, then their associated proper length metrics are quasi-isometric.

The following concept is due to Gromov \cite{G1}.

\begin{defn} The asymptotic dimension of a proper metric space $X$ is the smallest integer $n$ such that for every $r>0$,
there exists a uniformly bounded cover $\{ U_i\}$ for which the number of $U_i$ intersecting each $r$ ball $B(x, r)$ is at most $n+1$.
\end{defn}

For example,  the asymptotic dimension of ${\mathbb Z}^n$ is $n$ and the asymptotic dimension of the free group ${\mathbb F}_n$ with $n$ generators is $1$. The asymptotic dimension is invariant under coarse equivalence. The Lie group $GL(n, {\mathbb R})$ with a left invariant Riemannian metric is quasi-isometric to $T(n, {\mathbb R})$, the subgroup of invertible upper triangular matrices.
By permanence properties of asymptotic dimension [BD1], we know that the solvable group $T(n, {\mathbb R})$ has finite asymptotic dimension. As a consequence, every countable discrete subgroup of $GL(n, {\mathbb R})$ has finite asymptotic dimension
(as a metric space with a proper length metric).
More generally one can prove that every discrete subgroup of an almost connected Lie group has finite asymptotic dimension 
(a Lie group is said to be almost connected if the number of its connected components is finite). 
Gromov's hyperbolic groups  have finite asymptotic dimension \cite{Roe2}.
Mapping class groups also have finite asymptotic dimension \cite{BBF}.

In \cite{Y1}, G. Yu developed a quantitative  operator K-theory to prove the following theorem.

\begin{thm} The Novikov conjecture holds if the fundamental group of a manifold 
has finite asymptotic dimension.
\end{thm}

The basic idea of the proof is that the finiteness of asymptotic dimension allows us to develop an algorithm to compute K-theory in a quantitative way. This strategy has found applications to topological rigidity of manifolds \cite{GTY}.

The following concept  of Gromov makes precise of the idea of drawing a good picture of a metric space in a Hilbert space.

\begin{defn}(Gromov): Let $X$ be a metric space and  $H$ be a Hilbert space.  A map $f : X \rightarrow
H$ is said to be a coarse embedding if there exist non-decreasing
functions $\rho_1$ and $\rho_2$ on $[0,  \infty )$ such that
\item {(1)} $\rho_1 (d(x,y)) \leq d_H (f(x), f(y)) \leq \rho_2 (d(x,y))$
for all $x,y \in X$;
\item {(2)} $\lim_{r \rightarrow +\infty} \rho_1 (r) = + \infty$.

\end{defn}

Coarse embeddability of a countable group is independent of the choice of proper length metrics. 
Examples of groups coarsely embeddable into Hilbert space include groups acting properly and isometrically on a Hilbert space (in particular amenable groups \cite{BCV}), groups with Property A \cite{Y2}, 
 countable subgroups of connected Lie groups \cite{GHW}, 
 hyperbolic groups \cite{S},  groups with finite asymptotic dimension, Coxeter groups \cite{DJ}, mapping class groups \cite{Ki, Ha}, 
 and semi-direct products of groups of the above types.

The following theorem unifies the above theorems.

\begin{thm}\label{thm:coarsenov} The Novikov conjecture holds if the fundamental group of a manifold 
is coarsely embeddable into Hilbert space.
\end{thm}

Roughly speaking, this theorem says if we can draw a good picture of the fundamental group in a Hilbert space, then we can recognize the manifold at an infinitesimal level. This theorem was proved by G. Yu when the classifying space of the fundamental group has the homotopy type of a finite CW complex \cite{Y2} and this finiteness condition was removed by N. Higson \cite{H}, Skandalis-Tu-Yu \cite{STY}. The original proof of the above result makes heavy use of infinite diimensional analysis.  More recently, R. Willett and G. Yu found a relatively elementary proof within the framework of basic operator K-theory
\cite{WiY}.

E. Guentner, N. Higson and S. Weinberger proved the beautiful theorem that
linear groups are coarsely embeddable into Hilbert space \cite{GHW}.  Recall that a group is called linear if it is a subgroup of $GL(n, k)$ for some field $k$. The following theorem follows as a consequence \cite{GHW}.

\begin{thm} The Novikov conjecture holds  if the fundamental group of a manifold is a linear group.
\end{thm}

More recently, Bestvina-Guirardel-Horbez proved that 
$\operatorname{Out}(F_n)$, the outer automorphism groups of the free groups,
is coarsely embeddable into Hilbert space. This implies the following theorem \cite{BGH}.

\begin{thm} The Novikov conjecture holds   if the fundamental group of a manifold is a subgroup of  $\operatorname{Out}(F_n)$.
\end{thm}

We have the following open question.

\begin{open} Is  every countable subgroup of the diffeomorphism group of the circle coarsely embeddable into Hilbert space?
\end{open}

Let $\mathfrak{E}$ be the smallest class of groups which include all groups coarsely embeddable into Hilbert space and is closed under direct limit.  Recall that if $I$ is a directed set and $\{G_i\}_{i\in I}$ is a direct system of groups
over $I$, then we can define the direct limit $\lim G_i$. We emphasize that here
the homomorphism $\phi_{ij}: G_i\rightarrow G_j$ for $i\leq j$ is not necessarily injective. 

The following result is a consequence of Theorem $\ref{thm:coarsenov}$.

\begin{thm} The Novikov conjecture holds   if the fundamental group of a manifold is in the class $\mathfrak{E}$.
\end{thm}

The following open question is a challenge to geometric group theorists.

\begin{open} Is there any countable group not in the class $\mathfrak{E}$?
\end{open}

We mention that the Gromov monster groups are in the class $\mathfrak{E}$ \cite{G2, G3, AD, O}.

Next we shall discuss the connection of the Novikov conjecture with geometry of Banach spaces.

\begin{defn}  A  Banach space $X$ is said to have Property (H) if there exist  an increasing sequence of finite dimensional subspaces
$\{ V_n\} $ of $X$ and an increasing  sequence of finite dimensional subspaces $\{ W_n\} $ of a Hilbert space such that
\begin{enumerate}
\item[(1)]
 $V=\cup_n V_n$ is dense in $X$,

\item[(2)]  if  $W=\cup_n W_n$,  $S(V)$ and $S(W)$ are respectively the unit spheres of $V$ and $W$, then there exists a uniformly continuous map $\psi: S(V)\rightarrow S(W)$ such that
the restriction of $\psi$ to $S(V_n)$ is a homeomorphism (or more generally a degree one map) onto $S(W_n)$ for each $n.$

\end{enumerate}

\end{defn}

As an example, let $X$ be the  Banach space $l^p(\mathbb{N})$ for some  $p\geq 1$.
Let $V_n$  and $W_n$  be respectively the subspaces of $l^p(\mathbb{N})$  and $l^2(\mathbb{N})$ consisting of all sequences whose coordinates are zero after the  $n$-th terms.
We define a map $\psi$ from $S(V)$ to $S(W)$ by
 $$\psi(c_1, \cdots, c_k, \cdots)= (c_1 |c_1|^{p/2-1}, \cdots, c_k |c_k|^{p/2-1},\cdots). $$
 $\psi$ is called the Mazur map.
It is not difficult to verify that $\psi$ satisfies the conditions in the definition of Property (H).
For each $ p\geq 1$, we can similarly prove that $C_p$, the Banach space of all Schatten $p$-class operators on a Hilbert space, has Property (H).  

Kasparov and Yu proved the following.

\begin{thm} The Novikov conjecture holds if the fundamental group of a manifold 
is coarsely embeddable into a Banach space with Property (H).
\end{thm}

Let $c_0$ be the Banach space consisting of  all sequences of real numbers that are convergent to $0$ with the supremum norm.

\begin{open} Does  the Banach space $c_0$ have Property (H)?
\end{open}

A positive answer to this question would imply the Novikov conjecture since every countable group admits a coarse embedding into $c_0$ \cite{BG}.

A less ambitious question is the following.

\begin{open}
Is every countable subgroup of the diffeomorphism group of a compact smooth manifold  coarsely embeddable into $C_p$ for some $p\geq 1$?
\end{open}

For each $p>q\geq 2$, it is also an open question to construct a bounded geometry space
which is coarsely embeddable into $l^p(\mathbb{N})$ but not $l^q(\mathbb{N})$.  Beautiful  results in \cite{JR} and \cite{MN}
indicate that  such a construction should be possible. Once such a metric space is constructed, the next natural question is to construct countable groups which coarsely contain such a metric space. These groups would be from another universe and would be different from any group we currently know.

\subsection{Gelfand-Fuchs classes, the group of volume preserving diffeomorphisms, Hilbert-Hadamard spaces}
\label{}

In this subsection, we give an overview on the work of A. Connes, Connes-Gromov-Moscovici on the Novikov conjecture for Gelfand-Fuchs classes and the recent work of  Gong-Wu-Yu on the Novikov conjecture for groups acting properly and isometrically on a Hilbert-Hadamard spaces and for any geometrically discrete subgroup of the group of volume preserving diffeomorphisms of a compact smooth manifold.

A. Connes proved the following deep theorem on the Novikov conjecture \cite{C1}.

\begin{thm} 
The Novikov conjecture holds for higher signatures associated to the Gelfand-Fuchs cohomology classes of a subgroup of the group of diffeomorphisms of a compact smooth manifold.
\end{thm}

The proof of this theorem uses the full power of noncommutative geometry \cite{C}.

\begin{open} Does the Novikov conjecture hold for any subgroup of the group of diffeomorphisms of a compact smooth manifold?
\end{open}

Motivated in part by this open question, S. Gong, J. Wu and G. Yu proved the following theorem \cite{GWY}.

\begin{thm}\label{thm:HHnov} 
The Novikov conjecture holds for groups acting properly and isometrically on an admissible Hilbert-Hadamard space. 
\end{thm}

Roughly speaking,  Hilbert-Hadamard spaces are (possibly infinite dimensional) simply connected spaces with non-positive curvature. We will give a precise definition a little later.
We say that a  Hilbert-Hadamard  space $M$ is \textit{admissible} if it has a sequence of subspaces $M_n$, whose union is dense in $M$, such that each $M_n$, seen with its inherited metric from $M$, is isometric to a finite-dimensional Riemannian manifold. Examples of admissible Hilbert-Hadamard spaces
include all simply connected and non-positively curved Riemannian manifold, the Hilbert space, and certain infinite dimensional symmetric spaces. Theorem 4.3
can be viewed as a generalization of both Theorem 2.1 and Theorem 3.1.

Infinite dimensional symmetric spaces are often naturally admissible Hilbert-Hadamard spaces. One such an example is 
\[L^2(N,\omega, \operatorname{SL}(n, \mathbb{R})/\operatorname{SO}(n)),\]
where $N$ is a compact smooth manifold with a given volume form $\omega$.
This infinite-dimensional symmetric space
 is defined to be the completion of  the space of all smooth maps 
from $N$ to $X=\operatorname{SL}(n, \mathbb{R})/\operatorname{SO}(n)$
with respect to the following distance:
		$$d(\xi, \eta) = \left( \int_N (d_X(\xi(y),\eta(y)))^2 \, d\omega(y) \right)^{\frac{1}{2}},$$ where $d_X$ is the standard Riemannian metric on the symmetric space $X$ and $\xi$ and $\eta$ are two smooth maps from $N$ to $X$. This space can be considered as  the space of $L^2$-metrics on $N$ with the given volume form 
$\omega$  and is a Hilbert-Hadamard space.
With the help of this infinite dimensional symmetric space, the above theorem can be applied to study the Novikov conjecture for geometrically discrete subgroups of the group of volume preserving diffeomorphisms on such a manifold.

The key ingredients of the proof for Theorem $\ref{thm:HHnov}$ include  a construction of a $C^\ast$-algebra modeled after the Hilbert-Hadamard  space, a deformation technique for the isometry group of the Hilbert-Hadamard  space and its corresponding actions on K-theory, and a KK-theory with real coefficient developed by Antonini, Azzali, and Skandalis \cite{AAS}. 

Let $\operatorname{Diff}(N,\omega)$ denote the group of volume preserving diffeomorphisms on a compact orientable smooth manifold $N$ with a given volume form $\omega$. In order to define the concept of geometrically discrete subgroups of $\operatorname{Diff}(N,\omega)$, let us fix a Riemannian metric on $N$  with the given volume  $\omega$ and define a length function $\lambda$ on  $\operatorname{Diff}(N,\omega)$ by:
\[ \lambda_+(\varphi) = \left(\int_N (\log(\|D\varphi\|))^2 d\omega\right)^{1/2}\]
and
\[ \lambda(\varphi) =  \max\left\{\lambda_+(\varphi),\lambda_+(\varphi^{-1})\right\} \; \]
 for all $\varphi \in \operatorname{Diff}(N,\omega)$, 
where $D\varphi$ is the Jacobian of $\varphi$, and the norm $\|\cdot \|$ denotes the operator norm, computed using the chosen Riemannian metric on $N$.

\begin{defn}
A subgroup $\Gamma$ of $\operatorname{Diff}(N,\omega)$ is said to be a geometrically discrete subgroup 
if $\lambda(\gamma) \rightarrow \infty$ when $\gamma \to \infty$ in $\Gamma$, i.e. for any $R>0$, there exists a finite subset $F\subset \Gamma$ such that $\lambda(\gamma)\geq R$ if $\gamma \in \Gamma \setminus F$. 
\end{defn}

Observe that although the length function $\lambda$ depends on our choice of the Riemannian metric, the above notion of geometric discreteness does not. Also notice that if $\gamma$ preserves the Riemannian metric we chose, then $\lambda(\gamma)=0$. This suggests that the class of geometrically discrete subgroups of $\operatorname{Diff}(N,\omega)$  does not intersect with the class of groups of isometries. Of course,  we already know the Novikov conjecture for any group of isometries on a compact Riemannian  manifold. This, together with the following result, gives an optimistic perspective on the open question on the Novikov conjecture for groups of volume preserving diffeomorphisms.

\begin{thm} 
	Let $N$ be a compact smooth manifold with a given volume form $\omega$, and let $\operatorname{Diff}(N,\omega)$ be the group of all volume preserving diffeomorphisms of $N$. The  Novikov conjecture holds for any geometrically discrete subgroup of $\operatorname{Diff}(N,\omega)$. 

\end{thm}

Now let us give a precise definition of Hilbert-Hadamard space. We will first recall the concept of CAT(0) spaces. 
Let $X$ be a geodesic metric space.
Let $\Delta $ be a triangle in $X$ with geodesic segments as its sides. $\Delta$
  is said to satisfy the CAT(0) inequality if there is a comparison triangle 
  $\Delta '$ in Euclidean space, with sides of the same length as the sides of 
  $ \Delta$ , such that distances between points on $\Delta $ are less than or equal to the distances between corresponding points on $ \Delta '.$
  The geodesic metric space $X$ is said to be a CAT(0) space if every geodesic triangle satisfies the CAT(0) inequality. 

Let $X$ be a geodesic metric space. For three distinct points $x,y,z \in X$, we define the comparison angle $\widetilde{\angle} xyz$ to be
\[\widetilde{\angle} xyz = \arccos\left(\frac{d(x,y)^2 + d(y,z)^2 - d(x,z)^2}{2d(x,y)d(y,z)}\right).\]
In other words, $\widetilde{\angle} xyz$ can be thought of as the angle at $y$ of the comparison triangle $\Delta xyz$ in the Euclidean plane. 


Given two nontrivial geodesic paths $\alpha$ and $\beta$ emanating from a point $p$ in $X$, meaning that $\alpha(0) = \beta(0) = p$, we define the angle between them, $\angle(\alpha,\beta)$, to be 
\[\angle(\alpha,\beta) = \lim_{s,t \to 0} \widetilde{\angle} (\alpha(s),p,\beta(t)) \; ,\]
provided that the limit exists. For CAT(0) spaces, since the $\widetilde{\angle} (\alpha(s),p,\beta(t))$ decreases with $s$ and $t$, the angle between any two geodesic paths emanating from a point is well defined. These  angles satisfy the triangle inequality.

For a point $p \in X$, let $\Sigma_p'$ denote the metric space induced from the space of all geodesics emanating from $p$ equipped with the pseudometric of angles, that is, for geodesics $\alpha$ and $\beta$, we define $d(\alpha,\beta) =\angle(\alpha,\beta)$. Note, in particular, from our definition of angles, that $d(\alpha,\beta) \leq \pi$ for any geodesics $\alpha$ and $\beta$.

We define  $\Sigma_p$ to be the completion of $\Sigma_p'$ with respect to the distance $d$. 
The \textit{tangent cone} $K_p$ at a point $p$ in $X$ is then defined to be a metric space which is, as a topological space, the cone of $\Sigma_p$. That is, topologically
\[K_p \simeq \Sigma_p \times [0,\infty)/\Sigma_p \times \{0\}.\]

The metric on it is given as follows. For two points $p,q \in K_p$ we can express them as $p=[(x,t)]$ and $q = [(y,s)]$. Then the metric is given by
\[d(p,q) = \sqrt{t^2+s^2-2st\cos(d(x,y))}.\]
The distance is what the distance would be if we went along geodesics in a Euclidean plane with the same angle between them as the angle between the corresponding directions in $X$.

The following definition is inspired by \cite{FS}.

\begin{defn}\label{defn:hhs}
	A \emph{\hhs} is a complete geodesic CAT(0) metric space (i.e., an Hadamard space) all of whose tangent cones are isometrically embedded in Hilbert spaces. 
\end{defn}

	Every connected and simply connected \emph{Riemannian-Hilbertian manifold with non-positive sectional curvature} is a separable {\hhs}. In fact, a Riemannian manifold without boundary is a {\hhs} if and only if it is complete, connected, and simply connected, and has nonpositive sectional 	curvature. 
We remark that a CAT(0) space $X$ is always uniquely geodesic. 

Recall that a subset of a geodesic metric space is called \emph{convex} if it is again a geodesic metric space when equipped with the restricted metric. We observe that a closed convex subset of a {\hhs} is itself a {\hhs}.

\begin{defn}
	A separable {\hhs} $M$ is called \emph{admissible} if there is a sequence of convex subsets isometric to finite-dimensional Riemannian manifolds, whose union is dense in $M$. 
\end{defn}

The notion of {\hhs}s is more general than simply connected  Riemannian-Hilbertian space with non-positive sectional curvature. For example, the infinite dimensional symmetric space $L^2(N,\omega, \operatorname{SL}(n, \mathbb{R})/\operatorname{SO}(n))$
is a {\hhs} but not a Riemannian-Hilbertian space with non-positive sectional curvature.

\section{Secondary invariants for Dirac operators and applications }\label{sec:secondary}

We have been mainly concerned with the primary invariants, i.e. the higher index invariants, till now. Starting this section, we shall shift our focus to secondary invariants. We try to keep the discussion relatively self-contained, which hopefully will  give a better sense of some of the more recent development on secondary invariants.

   In this section, we introduce a secondary invariant 
for Dirac operators on manifolds with positive scalar curvature and apply the invariant to measure the size of the moduli space of Riemmanian metrics with positive scalar curvature on a given spin manifold. 

We carry out the construction in the odd dimensional
case. The even dimensional case is similar.
Suppose  that $X$ is an odd dimensional complete spin manifold without
boundary and we fix a spin structure on $X$. Assume that there is a
discrete group $\Gamma$ acting on $X$ properly and cocompactly by
isometries. In addition, we assume the action of $\Gamma$ preserves
the spin structure on $X$. A typical such example comes from a Galois
cover $\widetilde{M}$ of a closed spin manifold $M$ with $\Gamma$
being the group of deck transformations.

Let $S$ be the spinor bundle over $M$ and $D$ be the associated
Dirac operator on~$X$. Let $H_{X} = L^{2}(X, S) $ and
\begin{eqnarray*} F = D( D^{2} + 1)^{-1/2}.
\end{eqnarray*}
 $(H_{X}, F)$ defines a class in $K_{1}^{\Gamma}(X)$. Note that
$F$ lies in the multiplier algebra of $C^{\ast}(X)^{\Gamma}$, since
$F$ can be approximated by elements of finite propagation in the
multiplier algebra of $C^{\ast}(X)^{\Gamma}$. As a result, we can
directly work with\footnote{In other words, there is no need to pass to
	the operator $\mathcal{F}$ or $\mathcal{F}(t)$ as in the general case.}
%
\begin{equation}\label{eq:path2}
F(t) = \sum_{j} ((1 - (t-n)) \phi_{n, j}^{1/2} F \phi_{n,j}^{1/2} +
(t-n) \phi_{n+1, j}^{1/2} F \phi_{n+1, j}^{1/2})
\end{equation}
for $t\in[n, n+1]$. The same index construction as before defines the index
class and the local index class of $(H_{X}, F)$. We shall denote them
by $\text{\textup{Ind}}(D)\in K_{1}(C^{\ast}(X)^{\Gamma})$ and $\text{\textup
	{Ind}}_{L}(D)\in K_{1}(C_{L}^{\ast}(X)^{\Gamma})$ respectively.

Now suppose in addition $X$ is endowed with a complete Riemannian
metric $g$ whose scalar curvature $\kappa$ is positive everywhere,
then the associated Dirac operator in fact naturally defines a class in
$K_{1}(C_{L, 0}^{\ast}(X)^{\Gamma})$. Indeed, recall that
\begin{eqnarray*} D^{2} = \nabla^{\ast} \nabla+ \frac{\kappa}{4},
\end{eqnarray*}
where $\nabla: C^{\infty}(X, S) \to C^{\infty}(X, T^{\ast} X\otimes
S)$ is the associated connection and $\nabla^{\ast}$ is the adjoint
of $\nabla$. If $\kappa>0$, then it follows immediately that $ D$ is
invertible. So, instead of $D(D^{2}+1)^{-1/2}$, we can use
\begin{eqnarray*} F := D|D|^{-1}.
\end{eqnarray*}
Note that $\frac{F+1 }{2}$ is a genuine projection. Define $F(t)$ as in
formula $\eqref{eq:path2}$, and define $q(t) := \frac{F(t) +
	1}{2}$. We form the path
of unitaries $u(t) = e^{2\pi i q(t)}, 0\leq t < \infty$, which defines an element in $
(C_{L}^{\ast}(X)^{\Gamma})^{+}$. Notice that $u(0) = 1$. So this path
$u(t), 0\leq t < \infty$, in fact lies in $(C_{L,0}^{\ast}(X)^{\Gamma
})^{+}$, therefore defines a class in $K_{1}(C_{L, 0}^{\ast
}(X)^{\Gamma})$.

Let us now define the higher rho invariant. It was first introduced by
Higson and Roe \cite{Roe1, HR3}. Our formulation is
slightly different from that of Higson and Roe. The equivalence of the
two definitions was shown in \cite{XY}.
%
\begin{defns}
	The higher rho invariant $\rho(D, g)$ of the pair $(D, g)$ is defined
	to be the $K$-theory class $[u(t)]\in K_{1}(C_{L, 0}^{\ast}(X)^{\Gamma})$.
\end{defns}

The definition of higher rho invariant in the even dimensional case is similar, where one needs to work with the
natural $\mathbb{Z}/2\mathbb{Z}$-grading on the spinor bundle.

Next we shall apply the higher rho invariant to estimate the size of the moduli space of Riemannian metrics with positive scalar curvature on a given spin manifold.
Let $M$ be a closed smooth manifold. Suppose that $M$ carries a metric of
	positive scalar curvature. It is well known that the space of all
	Rimennian metrics on $M$ is contractible, hence topologically trivial.
	To the contrary, the space of all positive scalar curvature metrics on
	$M$, denoted by $\mathcal{R}^{+}(M)$, often has very nontrivial
	topology. In particular, $\mathcal{R}^{+}(M)$ is often not connected and
	in fact has infinitely many connected components \cite{BoG, LP, PS, RS}.  For example, by using the
	Cheeger--Gromov $L^{2}$-rho invariant  and Lott's
	delocalized eta invariant,  Piazza and Schick
	showed that $\mathcal{R}^{+}(M)$ has infinitely many connected
	components, if $M$ is a closed spin manifold with $\dim M = 4k+3 \geq
	5$ and $\pi_{1}(M)$ contains torsion \cite{PS}.

Following Stolz \cite{St}, Weinberger and Yu
	introduced an abelian group $P(M)$ to measure the size of the space of
	positive scalar curvature metrics on a manifold $M$ \cite{WY}. In
	addition, they used the finite part of $K$-theory of the maximal group
	$C^{\ast}$-algebra $C^{\ast}_{\max}(\pi_{1}(M))$ to give a lower
	bound of the rank of $P(M)$. A special case of their theorem states
	that the rank of $P(M)$ is $\geq1$, if $M$ is a closed spin manifold
	with $\dim M = 2k+1\geq5$ and $\pi_{1}(M)$ contains torsion. In particular, this implies the above
	theorem of Piazza and Schick.
	
For convenience of the reader, we recall the definition of the abelian group $P(M)$.
Let $M$ be an oriented smooth closed manifold with $\dim M\geq 5$ and its fundamental group $\pi_1(M) = \Gamma$. Assume that $M$ carries a metric of positive scalar curvature. We denote it by $g_M$.  Let $I$ be the closed interval $[0, 1]$. Consider the connected sum $(M\times I)\sharp (M\times I)$, where the connected sum is performed away from the boundary of $M\times I$. Note that $\pi_1\big((M\times I)\sharp(M\times I)\big) = \Gamma\ast \Gamma$ the free product of two copies of $\Gamma$. 

\begin{defns}
We define the generalized connected sum $(M\times I)\natural(M\times I)$ to be the manifold obtained from $(M\times I)\sharp(M\times I)$ by removing the kernel of the homomorphism $\Gamma\ast \Gamma\to \Gamma$ through surgeries away from the boundary.
\end{defns}

Note that $(M\times I)\natural(M\times I)$ has four boundary components, two of them being $M$ and the other two being $-M$, where $-M$ is the manifold $M$ with its reversed orientation. 
Now suppose $g_1$ and $g_2$ are two positive scalar curvature metrics on $M$. We endow one boundary component $M$ with $g_M$, and endow the two $-M$ components with $g_1$ and $g_2$. Then by the Gromov-Lawson and Schoen-Yau surgery theorem for positive scalar curvature metrics \cite{GL, SY}, there exists a positive scalar curvature metric on $(M\times I)\natural (M\times I)$ which is a product metric near all boundary components. In particular, the restriction of this metric on the other boundary component $M$ has positive scalar curvature. We denote this metric on $M$ by $g$.     

\begin{defns}
Two positive scalar curvature metrics $g$ and $h$ on $M$ are concordant if there exists a positive scalar curvature metric  on $M\times I$ which is a product metric near the boundary and restricts to $g$ and $h$ on the two boundary components respectively. 
\end{defns}

One can in fact show that if $g$ and $g'$ are two positive scalar curvature metrics on $M$ obtained from the same pair of positive scalar curvature metrics $g_1$ and $g_2$ by the above procedure, then $g$ and $g'$ are concordant \cite{WY}. 

\begin{defns}
Fix a positive scalar curvature metric $g_M$ on $M$.  Let $P^+(M)$ be the set of all concordance classes of positive scalar metrics on $M$. Given $[g_1]$ and $[g_2]$ in $P^+(M)$, we define the sum of $[g_1]$ and $[g_2]$ (with respect to $[g_M]$) to be $[g]$ constructed from the procedure above. Then it is not difficult to verify that $P^+(M)$ becomes an abelian semigroup under this addition. We define the abelian group $P(M)$ to be the Grothendieck group of $P^+(M)$.  
\end{defns}

	
	Recall that the
	group of diffeomorphisms on $M$, denoted by $\text{\textup{Diff}}(M)$, acts on
	$\mathcal{R}^{+}(M)$ by pulling back the metrics. The moduli space of
	positive scalar curvature metrics is defined to be the quotient space
	$\mathcal{R}^{+}(M)/\text{\textup{Diff}}(M)$. Similarly, $\text{\textup{Diff}}(M)$
	acts on the group $P(M)$ and we denote the coinvariant of the action by
	$\widetilde{P}(M)$. That is, $\widetilde{P}(M) = P(M)/P_{0}(M)$ where
	$P_{0}(M)$ is the subgroup of $P(M)$ generated by elements of the form
	$[x] - \psi^{\ast}[x]$ for all $ [x]\in P(M)$ and all $\psi\in
	\text{\textup{Diff}}(M)$. We call $\widetilde{P}(M)$ the moduli group of
	positive scalar curvature metrics on $M$. It measures the size of the
	moduli space of positive scalar curvature metrics on $M$. The following
	conjecture gives a lower bound for the rank of the abelian group
	$\widetilde{P}(M)$.
	\begin{conj}\label{con:rank}
		Let $M$ be a closed spin manifold with $\pi_{1}(M) = \Gamma$ and
		$\dim M = 4k-1 \geq5$, which carries a positive scalar curvature
		metric. Then the rank of the abelian group $\widetilde{P}(M)$ is $\geq
		N_{\text{\textup{fin}}}(\Gamma)$, where $N_{\text{\textup{fin}}}(\Gamma)$ is the
		cardinality of the following collection of positive integers:
		\begin{eqnarray*} \big\{ d \in\mathbb{N}_{+} \mid\exists\gamma\in
			\Gamma\ \text{\textup{such that}}\ \text{\textup{order}}(\gamma) = d \ \text{\textup{and}}
			\ \gamma\neq e\big\}.
		\end{eqnarray*}
	\end{conj}

	In \cite{XY1}, we apply the higher rho invariants of the Dirac operator to prove the following result.

	\begin{thm}\label{modulipsc}
		Let $M$ be a closed spin manifold which carries a positive scalar
		curvature metric with $\dim M = 4k-1 \geq5$. If the fundamental group
		$\Gamma= \pi_{1}(M)$ of $M$ is strongly finitely embeddable into
		Hilbert space, then the rank of the abelian group $\widetilde{P}(M)$ is
		$\geq N_{\text{\textup{fin}}}(\Gamma)$.
	\end{thm}
	
	To prove this theorem, we need index theoretic invariants that are insensitive to the action of the
	diffeomorphism group. The index theoretic techniques used
	in \cite{WY}, for example, do not produce such
	invariants. The key idea of the proof is 
	that the higher rho invariant remains unchanged in a certain $K$-theory
	group under the action of the diffeomorphism group, allowing us to
	distinguish elements in $\widetilde{P}(M)$. 
	
	We now recall the concept of strongly finite embeddability into Hilbert space for
	groups \cite{XY1}. This concept  is a stronger version of the notion of finite embeddability into
	Hilbert space introduced  in \cite{WY}, a concept  more flexible than the notion of coarse embeddability.

\begin{defns}
	A countable discrete group $\Gamma$ is said to be finitely embeddable
	into Hilbert space $H$ if for any finite subset $F\subseteq\Gamma$,
	there exist a group $\Gamma'$ that is coarsely embeddable into $H$ and
	a map $\phi: F\to\Gamma'$ such that
	\begin{enumerate}[(2)]
		\item[(1)] if $\gamma, \beta$ and $\gamma\beta$ are all in $F$,
		then $\phi(\gamma\beta) = \phi(\gamma) \phi(\beta)$;
		\item[(2)] if $\gamma$ is a finite order element in $F$, then $\text{\textup
			{order}}(\phi(\gamma)) = \text{\textup{order}}(\gamma)$.
	\end{enumerate}
\end{defns}

As mentioned above, Weinberger and Yu proved that Conjecture $\ref{con:rank}$ holds for all groups that are finitely embeddable into
Hilbert space \cite{WY}.

 If $g\in \Gamma$ has  finite order $d$, then we can define an idempotent in the group algebra ${\mathbb  Q} \Gamma$ by:
$$p_g= \frac{1}{d}(\sum_{k=1}^{d} g^k).$$

For the rest of this survey, we denote the maximal group $C^*$-algebra of $\Gamma$  by $C^\ast(\Gamma)$.

\begin{defns}
	We define $ K_0 ^{\fin}( C^*(\Gamma)) $,  called the finite part of  $K_0 ( C^*(\Gamma))$, to be the abelian subgroup of $K_0 ( C^*(G))$ generated by $[p_g]$ for all elements $g\neq e$ in $G$ with finite order.
	\end{defns} 
	
	We remark that rationally all representations of a finite group are induced from its finite cyclic subgroups \cite{Serre}. This explains that the finite part of $K$-theory, despite being constructed using only cyclic subgroups,  rationally contains all $K$-theory elements which can be constructed using finite subgroups.


\begin{defns}	Let $\mathcal
	{J}_{0}^{\text{\textup{fin}}}(C^{\ast}(\Gamma))$ be the abelian subgroup of
	$K_{0}^{\text{\textup{fin}}}(C^{\ast}(\Gamma))$ generated by elements
	$[p_{\gamma}] - [p_{\beta}]$ with $\text{\textup{order}}(\gamma) = \text{\textup
		{order}}(\beta)$.
	We define the reduced finite part of $K^{0}(C^{\ast}(\Gamma))$ to be
	\begin{eqnarray*} \widetilde{K}_{0}^{\text{\textup{fin}}}(C^{\ast}(\Gamma))
		= K_{0}^{\text{\textup{fin}}}(C^{\ast}(\Gamma))/\mathcal{J}_{0}^{\text{\textup
				{fin}}}(C^{\ast}(\Gamma)).
	\end{eqnarray*}
	\end{defns} 
	
	An argument in  \cite{WY} can be used to
	prove the following result, which plays a crucial role in the proof of Theorem 6.2.
	
\begin{prop}\label{prop:rank}	Let $\{\gamma_1, \cdots, \gamma_n\}$ be a collection of nontrivial elements (i.e. $\gamma_i\neq e$) with distinct finite order in $\Gamma$.
We define $\mathcal M_{\gamma_1, \cdots, \gamma_n}$ to be the abelian subgroup of $K_0^\fin(C^\ast(\Gamma))$ generated by $\{ [p_{\gamma_1}], \cdots, [p_{\gamma_n}]\}$.
	Let $\widetilde{\mathcal{M}}_{\gamma_{1}, \cdots, \gamma_{n}}$ be
	the image of ${\mathcal{M}}_{\gamma_{1}, \cdots, \gamma_{n}}$ in
	$\widetilde{K}_{0}^{\text{\textup{fin}}}(C^{\ast}(\Gamma))$. If $\Gamma$ is
	finitely embeddable into Hilbert space, then
	\begin{enumerate}[(2)]
		\item[(1)] the abelian group $\widetilde{\mathcal{M}}_{\gamma_{1},
			\cdots, \gamma_{n}}$ has rank $n$,
		\item[(2)] any nonzero element in $K_{0}^{\text{\textup{fin}}}(C^{\ast
		}(\Gamma))$ is not in the image of the assembly map
		\begin{eqnarray*} \mu: K_{0}^{\Gamma}(E\Gamma) \to K_{0}(C^{\ast
			}(\Gamma)),
		\end{eqnarray*}
		where $E\Gamma$ is the universal space for proper and free $\Gamma$-action.
	\end{enumerate}
\end{prop}

So one is led to the following conjecture.

\begin{conj}\label{genconj} Let $\Gamma$ be a countable discrete group.
	Suppose $\{\gamma_{1}, \cdots, \gamma_{n}\}$ is a collection of
	elements in $\Gamma$ with distinct finite orders and $\gamma_{i}\neq
	e$ for all $1\leq i\leq n$. Then
	\begin{enumerate}[\textit{(2)}]
		\item[\textit{(1)}] the abelian group $\widetilde{\mathcal{M}}_{\gamma_{1},
			\cdots, \gamma_{n}}$ has rank $n$,
		\item[\textit{(2)}] any nonzero element in $K_{0}^{\text{\textup{fin}}}(C^{\ast
		}(\Gamma))$ is not in the image of the assembly map
		\begin{eqnarray*} \mu: K_{0}^{\Gamma}(E\Gamma) \to K_{0}(C^{\ast
			}(\Gamma)),
		\end{eqnarray*}
		where $E\Gamma$ is the universal space for proper and free $\Gamma$-action.
	\end{enumerate}
\end{conj}


We are now ready to introduce the notion of strongly finitely
embeddability for groups. Since we are interested in the fundamental
groups of manifolds, all groups are assumed to be finitely generated in
the following discussion.

Let $\Gamma$ be a countable discrete group. Then any set of $n$
automorphisms of $\Gamma$, say, $\psi_{1}, \cdots, \psi_{n} \in
\text{\textup{Aut}}(\Gamma) $, induces a natural action of $F_{n}$ the free
group of $n$ generators on $\Gamma$. More precisely, if we denote the
set of generators of $F_{n}$ by $\{s_{1}, \cdots, s_{n}\}$, then we
have a homomorphism $F_{n}\to\text{\textup{Aut}}(\Gamma)$ by $s_{i}\mapsto
\psi_{i}$. This homomorphism induces an action of $F_{n}$ on $\Gamma
$. We denote by $\Gamma\rtimes_{\{\psi_{1}, \cdots, \psi_{n}\}}
F_{n}$ the semi-direct product of $\Gamma$ and $F_{n}$ with respect to
this action. If no confusion arises, we shall write $\Gamma\rtimes
F_{n}$ instead of $\Gamma\rtimes_{\{\psi_{1}, \cdots, \psi_{n}\}} F_{n}$.

\begin{defns}\label{def:strfe}
	A countable discrete group $\Gamma$ is said to be strongly finitely
	embeddable into Hilbert space $H$, if $\Gamma\rtimes_{\{\psi_{1},
		\cdots, \psi_{n}\}} F_{n}$ is finitely embeddable into Hilbert space
	$H$ for all $n\in\mathbb{N}$ and all $\psi_{1}, \cdots, \psi_{n}
	\in\text{\textup{Aut}}(\Gamma) $.
\end{defns}

We remark that all coarsely embeddable groups are strongly finitely
embeddable. Indeed, if a group $\Gamma$ is coarsely embeddable into
Hilbert space, then $\Gamma\rtimes_{\{\psi_{1}, \cdots, \psi_{n}\}
} F_{n}$ is coarsely embeddable (hence finitely embeddable) into
Hilbert space for all $n\in\mathbb{N}$ and all $\psi_{1}, \cdots,
\psi_{n} \in\text{\textup{Aut}}(\Gamma) $.

If a group $\Gamma$ has a torsion free normal subgroup $\Gamma
^{\prime}$ such that $\Gamma/\Gamma^{\prime}$ is residually finite,
then $\Gamma$ is strongly finitely embeddable into Hilbert space.
Indeed, recall that any finitely generated group has only finitely many
distinct subgroups of a given index.
Let $\Gamma_{m}$ be the intersection of all subgroups of $\Gamma$
with index at most $m$. Then $\Gamma/\Gamma_{m}$ is a finite group.
Moreover, for given $\psi_{1}, \cdots, \psi_{n} \in\text{\textup
	{Aut}}(\Gamma) $, the induced action of $F_{n}$ on $\Gamma$ descends
to an action of $F_{n}$ on $\Gamma/\Gamma_{m}$. In other words, we
have a natural homomorphism
\begin{eqnarray*} \phi_{m}\colon\Gamma\rtimes F_{n} \to(\Gamma
	/\Gamma_{m})\rtimes G_{m}
\end{eqnarray*}
where $G_{m}$ is the image of $F_{n}$ under the homomorphism $F_{n}\to
\text{\textup{Aut}}(\Gamma/\Gamma_{m})$. It follows that, for any finite set
$F\subseteq\Gamma$, there exists a sufficiently large $m$ such that
the map
\begin{eqnarray*} \phi_{m}: F\subset\Gamma\rtimes F_{n} \to(\Gamma
	/\Gamma_{m})\rtimes G_{m}
\end{eqnarray*}
satisfies
\begin{enumerate}[(2)]
	\item[(1)] if $\gamma, \beta$ and $\gamma\beta$ are all in $F$,
	then $\phi(\gamma\beta) = \phi(\gamma) \phi(\beta)$;
	\item[(2)] if $\gamma$ is a finite order element\footnote{Note that in
		this case, all finite order elements in $\Gamma\rtimes_{\{\psi_{1},
			\cdots, \psi_{n}\}} F_{n}$ come from $\Gamma$.} in $F$, then $\text{\textup
		{order}}(\phi(\gamma)) = \text{\textup{order}}(\gamma)$.
\end{enumerate}
Notice that $(\Gamma/\Gamma_{m})\rtimes
G_{m} $ is a finite group, which is obviously coarsely embeddable into Hilbert space. This shows that
$\Gamma$ is strongly finitely embeddable into Hilbert space.

To summarize, we see that the class of strongly finitely embeddable
groups includes all residually finite groups, virtually torsion free groups (e.g. $Out(F_{n})$), and  groups that coarsely embed into Hilbert space, where the latter contains all amenable groups and
Gromov's hyperbolic groups.

The notion of sofic groups is a generalization of amenable groups and
residually finite groups. It is an open question whether sofic groups
are (strongly) finitely embeddable into Hilbert space. Narutaka Ozawa,
Denise Osin and Thomas Delzant have independently constructed examples
of groups which are not finitely embeddable into Hilbert space. An
affirmative answer to the above question would imply that there exist
non-sofic groups.

By definition, strongly finite embeddability implies finite
embeddability. It is an open question whether the converse holds:

\begin{open}
	If a group is finitely embeddable into Hilbert space, then does it
	follow that the group is also strongly finitely embeddable into Hilbert space?
\end{open}

In fact, it was shown in \cite{WY} that Gromov's monster
groups and any group of analytic diffeomorphisms of an analytic
connected manifold fixing a given point are finitely embeddable into
Hilbert space. It is still an open question whether these groups are
strongly finitely embeddable into Hilbert space.

Now let us proceed to prove Theorem $\ref{modulipsc}$.  One of main ingredients of the proof is the following proposition\footnote{ Proposition $\ref{prop:gen}$ first appeared in \cite{WY}. The original statement in  \cite{WY} seems to contain a minor error when $d$ is even, the version we state in this survey and its proof can be found in \cite{XYZ}.
}, which, combined with a surgery technique \cite{GL, SY} and the relative higher index theorem \cite{B, XY2}, allows us to construct genuinely ``new" positive scalar curvature metrics from old ones. For a finite group $F$, an $F$-manifold $Y$ is called $F$-connected if the quotient $Y/F$ is connected. Let $\mathbb Z_d$ be the cyclic group of order $d$. 
\begin{prop} \label{prop:gen}
Given positive integers $d$ and $k$, there exist $\mathbb Z_d$-connected closed spin $\mathbb Z_d$-manifolds $\{Y_1, \cdots, Y_n\}$ with $\dim Y_i = 2k$  such that  

\begin{enumerate}[$(a)$]
\item the $\mathbb Z_d$-equivariant indices of the Dirac operators on $\{Y_1, \cdots, Y_n\}$ rationally generate $KO(\mathbb Z_d)\otimes \mathbb Q$,
\item $\mathbb Z_d$ acts on $Y_i$ freely except for finitely many fixed points.
\end{enumerate}

%
\end{prop}

 Let $M$ be a closed spin manifold with a positive scalar curvature metric $g_M$ and $\dim M\geq 5$  as before.   For each nontrivial finite order element $\gamma\in \Gamma$, one can construct a new positive scalar curvature metric $h_\gamma$ on $M$ such that the relative higher index $Ind_\Gamma(g_M, h_\gamma) = [p_\gamma] \in K_0(C^\ast(\Gamma)) $, where $p_\gamma = \frac{1}{d}\sum_{k=1}^d \gamma^k$ with $d = order(\gamma)$. The detailed construction will be given in the next paragraphs. Here let us  recall the definition of this relative higher index $Ind_\Gamma(g_M, h_\gamma)$.  We endow $M\times \mathbb R$ with the metric $g_t + (dt)^2$ where  $g_t$ is a smooth path of Riemannian metrics on $M$ such that 
\[ 
g_t = \begin{cases}
g_M \quad \textup{for $t\leq 0$, }\\
h_\gamma \quad \textup{for $t\geq 1$,} \\
\textup{any smooth path of metrics from $g_M$ to $h_\gamma$ for $0\leq t \leq 1$.}
\end{cases}\] 
Then $M\times \mathbb R$ becomes a complete Riemannian manifold with positive scalar curvature away from a compact subset. Denote by $D_{M\times \mathbb R}$ the corresponding Dirac operator on $M\times \mathbb R$ with respect to this metric. Then the higher index of $D_{M\times \mathbb R}$  is well-defined and is denoted by $Ind_\Gamma(g_M, h_\gamma)$ (cf. the discussion at the beginning of Section $\ref{sec:rhobdry}$ below).

Next  we shall describe  a construction of a new positive scalar curvature metric $h_\gamma$ on $M$  associated to a nontrivial finite order element $\gamma\in \Gamma$.
Let $\widetilde{M}$ be the universal cover of $M$. 
For each finite order element $g$ in $G$ with order $d$. By Proposition $\ref{prop:gen}$, there exist ${\mathbb Z}/d{\mathbb Z}$-connected  compact smooth spin ${\mathbb Z}/d{\mathbb Z}$-manifolds $\{N_1, \cdots, N_n\}$
 such that
the dimension of each $N_i$ is $4k$ and the sum of
the ${\mathbb Z}/d{\mathbb Z}$-equivariant indices of the Dirac operators on  $\{ N_1, \cdots, N_n\}$ is a nonzero multiple of
the trivial representation of ${\mathbb Z}/d{\mathbb Z}$.

Let $N_{g, l} = G\times_{{\mathbb Z}/d{\mathbb Z}}N_l, $ where ${\mathbb Z}/d{\mathbb Z} $ acts on $N_l$ as in Proposition $\ref{prop:gen}$ and ${\mathbb Z}/d{\mathbb Z}$ acts on $G$ by $[m]h=hg^m$ for all $h\in G$ and $[m]\in {\mathbb Z}/d{\mathbb Z}  $. Observe that $N_{g,l}$ is a $G$-manifold.

Let $\{g_1, \cdots, g_r\}$ be a collection of finite order elements such that \linebreak
 $\{[p_{g_1}], \cdots, [p_{g_r}]\}$ generates an abelian subgroup of
$K_0(C^\ast(G))$ with rank $r$. Let $N_{g_i}=\bigsqcup_{l=1}^{j_i}N_{g_i, l}$ be the disjoint union of all $G$-manifolds described as above. Let $I$ be the unit interval
$[0,1]$.
We first form a generalized  $G$-equivariant connected sum $(\widetilde{M}\times I)\natural N_{g_i}$ along a free $G$-orbit of each  $N_{g_i,l}$ and away from the boundary of $\widetilde{M}\times I$ as follows.
We first obtain a $G^{\ast j_i}$-equivariant  connected sum $(\widetilde{M}\times I) \sharp N_{g_i}$ along a free $G$-orbit of each  $N_{g_i, l}$ and away from the boundary of $\tilde{M}\times I$, where $G^{\ast j_i}$ is the free product of $j_i$ copies of $G$.  More precisely,  we inductively form the $G^{\ast j_i}$-equivariant connected sum
$( \cdots ((\widetilde{M}\times I) \sharp N_{g_i,1})\cdots )\sharp N_{g_i, j_i}$, where the equivariant connected sum is inductively taken along a free orbit
and away from the boundary.
We denote this space by $(\widetilde{M}\times I) \sharp N_{g_i}$.
  We then perform surgeries on  $(\widetilde{M}\times I) \sharp N_{g_i}$
 to obtain a $G$-equivariant cobordism between two copies of $G$-manifold $\widetilde{M}$.

For any positive scalar curvature metric $h$ on $M$, by \cite[Theorem 2.2]{RW},
the above cobordism gives us another positive scalar curvature metric $h_i$ on $M$.
Now   the relative higher index theorem \cite{B, XY} implies that the relative higher index of the Dirac operator $M\times {\mathbb R}$ associated to the positive scalar curvature  metrics of  $h_i$ and $g_M$ is $[p_{g_i}]$ in $K_0(C^\ast(G))$. As a consequence,   we know that
$\{ [h_1], \cdots, [h_r]\}$ generates an abelian subgroup of $P(M)$ with  rank $r$.

To summarize, one can construct distinct elements in $P(M)$ by surgery theory and the relative higher index theorem. Moreover, these elements are distinguished by their relative higher indices (with respect to $g_M$). 
However, to prove Theorem $\ref{modulipsc}$, that is, to show that these concordance classes of positive scalar curvature metrics remain distinct even after modulo the action of diffeomorphisms, we will need to use higher rho invariants (instead of relative higher indices) in an essential way. 

\begin{proof}[{Proof of Theorem $\ref{modulipsc}$}]

Consider the following short exact sequence
\[  0\to C^\ast_{L,0}(\widetilde M)^\Gamma \to C^\ast_L(\widetilde M)^\Gamma \to C^\ast(\widetilde M)^\Gamma \to 0\]
where $\widetilde M$ is the universal cover of $M$. It induces the following six-term long exact sequence:
\[ \xymatrix{K_0(C_{L,0}^\ast(\widetilde M)^\Gamma)\ar[r] &  K_0(C_L^\ast(\widetilde M)^\Gamma)  \ar[r]^{\mu_M} &  K_0(C^\ast(\widetilde M)^\Gamma) \ar[d]^{\partial} \\
K_1(C^\ast(\widetilde M)^\Gamma) \ar[u]& K_{1}(C^\ast_{L}(\widetilde M)^\Gamma) \ar[l] &  K_{1}(C_{L,0}^\ast(\widetilde M)^\Gamma) \ar[l]} \]
Recall that we have $ K_0(C_L^\ast(\widetilde M)^\Gamma) \cong K_0^\Gamma(\widetilde M) $ and $K_0(C^\ast(\widetilde M)^\Gamma) \cong K_0(C^\ast(\Gamma))$.

Fix a positive scalar curvature metric $g_M$ on $M$. For each finite order element $\gamma\in \Gamma$, we can construct a new positive scalar curvature metric $h_\gamma$ on $M$ such that the relative higher index $Ind_\Gamma(g_M, h_\gamma) = [p_\gamma] \in K_0(C^\ast(\Gamma)) $ as described as above. Let us still denote by $h_\gamma$ (resp. $g_M$) the metric on $\widetilde M$ lifted from the metric $h_\gamma$ (resp. $g_M$) on $M$. Let $\rho(D, h_\gamma)$ and $\rho(D, g_M)$ be the higher rho invariants for the pairs $(D, h_\gamma)$ and $(D, g_M)$, where $D$ is the Dirac operator on $\widetilde M$. 
Then we have
\begin{equation}\label{eq:bd}
\partial([p_\gamma]) = \partial\big(Ind_\Gamma(g_M, h_\gamma)\big) = \rho(D, h_\gamma) - \rho(D, g_M),
\end{equation}
(cf.  \cite{PS1, XY}).

One of the key points of the proof is to construct a certain group homomorphism on $\widetilde P(M)$ which can be used to distinguish elements in $\widetilde P(M)$. First, we define a map $ \varrho: P(M) \to  K_{1}(C_{L,0}^\ast(\widetilde M)^\Gamma)$  by 
\[\varrho(h) := \rho(D, h) - \rho(D, g_M)\]
 for all $h\in P(M)$. It follows from the definition of $P(M)$ and  \cite[Theorem 4.1]{XY} that the map $\varrho$ is a well-defined group homomorphism. Now recall that a diffeomorphism $\psi \in \text{\textup{Diff}}(M)$ induces a homomorphism
\[\psi_\ast: K_{1}(C_{L,0}^\ast(\widetilde M)^\Gamma)\to K_{1}(C_{L,0}^\ast(\widetilde M)^\Gamma).\]
Let $\mathcal I_1(C_{L,0}^\ast(\widetilde M)^\Gamma)$ be the subgroup of $K_{1}(C_{L,0}^\ast(\widetilde M)^\Gamma)$  generated by elements of the form $[x] - \psi_\ast[x]$ for all $[x]\in K_{1}(C_{L,0}^\ast(\widetilde M)^\Gamma)$ and all $\psi\in \text{\textup{Diff}}(M)$. We see that $\varrho$ descends to a group homomorphism
\[ \widetilde \varrho: \widetilde P(M) \to K_{1}(C_{L,0}^\ast(\widetilde M)^\Gamma)\big/\mathcal I_1(C_{L,0}^\ast(\widetilde M)^\Gamma). \]
To see this, it suffices to verify that
\[  \varrho(h) - \varrho(\psi^\ast(h)) \in \mathcal I_1(C_{L,0}^\ast(\widetilde M)^\Gamma) \]
for all $[h]\in P(M)$ and $\psi\in \text{\textup{Diff}}(M)$. Indeed, we have
\begin{align*}
 \varrho(h) - \varrho(\psi^\ast(h)) & = \rho(D, h) - \rho(D, g_M) - \big (\rho(D, \psi^\ast (h)) - \rho(D, g_M) \big) \\ 
 & = \rho(D, h) - \rho(D, \psi^\ast (h)) \\
 & = \rho(D, h) - \psi_\ast( \rho(D, h)) \in \mathcal I_1(C_{L,0}^\ast(\widetilde M)^\Gamma).
\end{align*}

We remark that it is crucial to use the higher rho invariant, instead of the relative higher index, to construct such a group homomorphism.  Let us explain the subtlety here.  Note that there is in fact a well-defined group homomorphism 
$ Ind_{rel}: P(M) \to  K_{0}(C^\ast(\Gamma))$ by $ Ind_{rel} (h)  = Ind_\Gamma(D; g_M, h).$
The well-definedness of $Ind_{rel}$ follows from the definition of $P(M)$ and the relative higher index theorem \cite{B, XY2}. 
However, in general, it is \textit{not} clear at all whether $Ind_{rel}$ descends to a group homomorphism $\widetilde P(M) \to  K_{0}(C^\ast(\Gamma))/\mathcal I_0(C^\ast(\Gamma)),$ 
where 	$\mathcal I_0(C^\ast(\Gamma))$ is the subgroup of $K_{0}(C^\ast(\Gamma))$  generated by elements of the form $[x] - \psi_\ast[x]$ for all $[x]\in K_{0}(C^\ast(\Gamma))$ and all $\psi\in \text{\textup{Diff}}(M)$.

Now for a collection of elements $\{\gamma_1, \cdots, \gamma_n\}$ with distinct finite orders, we consider the associated collection of positive scalar curvature metrics $\{h_{\gamma_1}, \cdots, h_{\gamma_n}\}$ as before. To prove the theorem, it suffices to show that for any collection of elements $\{\gamma_1, \cdots, \gamma_n\}$ with distinct finite orders, the elements 
\[ \widetilde \varrho(h_{\gamma_1}), \cdots, \widetilde \varrho(h_{\gamma_n})\]  
are linearly independent in $K_{1}(C_{L,0}^\ast(\widetilde M)^\Gamma)\big/\mathcal I_1(C_{L,0}^\ast(\widetilde M)^\Gamma)$. 

Let us assume the contrary, that is, there exist $[x_1], \cdots, [x_m]\in K_{1}(C_{L,0}^\ast(\widetilde M)^\Gamma) $ and $\psi_1, \cdots, \psi_m\in \text{\textup{Diff}}(M)$ such that
\begin{equation}\label{eq:van}
\sum_{i=1}^n c_i \varrho(h_{\gamma_i}) = \sum_{j=1}^{m} \big([x_j] - (\psi_{j})_\ast[x_j]\big),
\end{equation}
where $c_1, \cdots, c_n\in \mathbb Z$ with at least one $c_i\neq 0$.

We denote by $W$ the wedge sum of $m$ circles. The fundamental group $\pi_1(W)$ is the free group $F_m$ of $m$ generators $\{s_1, \cdots, s_m\}$. We denote the universal cover of $W$ by $\widetilde W$. Clearly, $\widetilde W$ is  the Cayley graph of $F_m$ with respect to the generating set $\{s_1, \cdots, s_m, s_1^{-1}, \cdots, s_m^{-1}\}$. Notice that $F_m$ acts on $M$ through the diffeomorphisms $\psi_1, \cdots, \psi_m$. In other words, we have a homomorphism  $F_m \to \text{\textup{Diff}}(M)$ by $s_i\mapsto \psi_i$. We define
\[  X = M \times_{F_m}\widetilde W. \]
Notice that $\pi_1(X) = \Gamma\rtimes_{\{\psi_1, \cdots, \psi_m\}} F_m$. Let us  write $\Gamma\rtimes F_m$ for $\Gamma\rtimes_{\{\psi_1, \cdots, \psi_m\}} F_m$, if no confusion arises.

Let $\widetilde X$ be the universal cover of $X$.  We have the following short exact sequence: 
\[  0\to C^\ast_{L,0}(\widetilde X)^{\Gamma\rtimes F_m} \to C^\ast_L(\widetilde X)^{\Gamma\rtimes F_m} \to C^\ast(\widetilde X)^{\Gamma\rtimes F_m} \to 0.\]
 Recall that  $ K_0(C_L^\ast(\widetilde X)^{\Gamma\rtimes F_m}) \cong K_0^{\Gamma\rtimes F_m}(\widetilde X) $ and $K_0(C^\ast(\widetilde X)^{\Gamma\rtimes F_m}) \cong K_0(C^\ast(\Gamma\rtimes F_m))$. So we have the following  six-term long exact sequence:
\begin{equation}\label{cd:bc}
\begin{gathered}
\xymatrix{K_0(C_{L,0}^\ast(\widetilde X)^{\Gamma\rtimes F_m})\ar[r] &  K_0^{\Gamma\rtimes F_m}(\widetilde X)  \ar[r] &  K_0(C^\ast(\Gamma\rtimes F_m)) \ar[d]^{\partial} \\
K_1(C^\ast(\Gamma\rtimes F_m)) \ar[u]& K_{1}^{\Gamma\rtimes F_m}(\widetilde X) \ar[l] &  K_{1}(C_{L,0}^\ast(\widetilde X)^{\Gamma\rtimes F_m}) \ar[l]} 
\end{gathered}
\end{equation}
Now recall the following Pimsner-Voiculescu exact sequence \cite{PV}:
\[ \xymatrixcolsep{6pc} \xymatrix{ \bigoplus_{j=1}^m K_0(C^\ast(\Gamma))\ar[r]^-{\sum_{j=1}^m 1-(\psi_j)_\ast} &   K_0(C^\ast(\Gamma))  \ar[r]^{i_\ast} &  K_0(C^\ast(\Gamma\rtimes F_m)) \ar[d] \\
K_1(C^\ast(\Gamma\rtimes F_m)) \ar[u]& K_{1}(C^\ast(\Gamma)) \ar[l] &  \bigoplus_{j=1}^m K_{1}(C^\ast(\Gamma)) \ar[l]_-{\sum_{j=1}^m 1-(\psi_j)_\ast}
} \]
where $(\psi_j)_\ast$ is induced by $\psi_j$ and $i_\ast$ is induced by the inclusion map of $\Gamma$ into $\Gamma\rtimes F_m$. Similarly, we also have the following two Pimsner-Voiculescu type exact sequences for $K$-homology and the $K$-theory groups of $C_{L,0}^\ast$-algebras in the diagram $\eqref{cd:bc}$ above.
\[ \xymatrixcolsep{6pc} \xymatrix{\bigoplus_{i=1}^m K_0^\Gamma(\widetilde M)\ar[r]^-{\sum_{j=1}^m 1-(\psi_j)_\ast} &  K_0^\Gamma(\widetilde M)  \ar[r]^{i_\ast} &  K_0^{\Gamma\rtimes F_m}(\widetilde X) \ar[d] \\
K_1^{\Gamma\rtimes F_m}(\widetilde X) \ar[u]& K_{1}^\Gamma(\widetilde M) \ar[l]&  \bigoplus_{i=1}^m K_{1}^\Gamma(\widetilde M) \ar[l]_-{\sum_{j=1}^m 1-(\psi_j)_\ast}  } \]
\[\xymatrixcolsep{5pc} \xymatrix{\bigoplus_{i=1}^m K_0(C_{L,0}^\ast(\widetilde M)^\Gamma)\ar[r]^-{\sum_{j=1}^m 1-(\psi_j)_\ast}  &  K_0(C_{L,0}^\ast(\widetilde M)^\Gamma)  \ar[r]^{i_\ast} &  K_0(C_{L,0}^\ast(\widetilde X)^{\Gamma\rtimes F_m}) \ar[d] \\
K_1(C_{L,0}^\ast(\widetilde X)^{\Gamma\rtimes F_m}) \ar[u]& K_{1}(C_{L,0}^\ast(\widetilde M)^\Gamma) \ar[l] & \bigoplus_{i=1}^m  K_{1}(C_{L,0}^\ast(\widetilde M)^\Gamma) \ar[l]_-{\sum_{j=1}^m 1-(\psi_j)_\ast} } \]
where again $(\psi_j)_\ast$ and $i_\ast$ are defined in the obvious way.

Combining these Pimsner-Voiculescu exact sequences together, we have the following commutative diagram: 
\begin{equation}\label{cd:pv}
\begin{gathered}
\scalebox{1}{ \xymatrix{  
& \vdots\ar[d]  & \vdots \ar[d] & \vdots\ar[d] & \\
 \ar[r] & \bigoplus_{j=1}^m K_0^{\Gamma}(\widetilde M) \ar[d] \ar[r]^-\sigma & K_0^{\Gamma}(\widetilde M) \ar[d] \ar[r]^-{i_\ast}  & K_0^{\Gamma\rtimes F_m}(\widetilde X) \ar[d]^{\mu} \ar[r] & \\
 \ar[r] &  \bigoplus_{j=1}^m K_0(C^\ast(\Gamma))\ar[r]^-{\sigma} \ar[d] &  K_0(C^\ast(\Gamma))  \ar[r]^-{i_\ast} \ar[d] &  K_0(C^\ast(\Gamma\rtimes F_m)) \ar[d]^{\partial_{\Gamma\rtimes F_m}} \ar[r] & \\
 \ar[r] &  \bigoplus_{j=1}^m K_{1}(C_{L,0}^\ast(\widetilde M)^{\Gamma}) \ar[r]^-\sigma \ar[d] & K_{1}(C_{L,0}^\ast(\widetilde M)^{\Gamma}) \ar[r]^-{i_\ast} \ar[d] & K_{1}(C_{L,0}^\ast(\widetilde X)^{\Gamma\rtimes F_m}) \ar[r] \ar[d] &  \\
& \vdots  & \vdots  & \vdots & \\
} }
\end{gathered}
\end{equation}
where $\sigma = \sum_{j=1}^m 1-(\psi_j)_\ast$. Notice that all rows and  columns are exact. 

Now on one hand,  if we pass Equation $\eqref{eq:van}$ to $K_{1}(C_{L,0}^\ast(\widetilde X)^{\Gamma\rtimes F_m})$ under the map $i_\ast$, then it follows immediately that 
\[  \sum_{k=1}^n c_k\cdot i_\ast[\varrho(h_{\gamma_k})] = 0 \textup{ in $K_{1}(C_{L,0}^\ast(\widetilde X)^{\Gamma\rtimes F_m})$}  , \]
where at least one $c_k\neq 0$. On the other hand, by assumption, $\Gamma$ is strongly finitely embeddable into Hilbert space. Hence $\Gamma\rtimes F_m$ is finitely embeddable into Hilbert space. By Proposition $\ref{prop:rank}$, we have the following. 
\begin{enumerate}[(i)]
\item $\{[p_{\gamma_1}], \cdots, [p_{\gamma_n}]\}$ generates a rank $n$ abelian subgroup of $K_0^\fin(C^\ast(\Gamma\rtimes F_m))$, since $\gamma_1, \cdots, \gamma_n$ have distinct finite orders. In other words, 
\[ \sum_{k=1}^n c_k [p_{\gamma_k}] \neq 0  \in K_0^\fin(C^\ast(\Gamma\rtimes F_m))\]
if at least one $c_k\neq 0$.
\item Every nonzero element in $K_0^\fin(C^\ast(\Gamma\rtimes F_m))$ is not in the image of the assembly map 
\[ \mu: K_0^{\Gamma\rtimes F_m}(E(\Gamma\rtimes F_m)) \to K_0(C^\ast(\Gamma\rtimes F_m)), \]
where $E(\Gamma\rtimes F_m)$ is the universal space for proper and free $\Gamma\rtimes F_m$-action. In particular, every nonzero element in $K_0^\fin(C^\ast(\Gamma\rtimes F_m))$ is not in the image of the map 
\[ \mu: K_0^{\Gamma\rtimes F_m}(\widetilde X) \to K_0(C^\ast(\Gamma\rtimes F_m)) \]
in diagram $\eqref{cd:pv}$. It follows that the map 
\[ \partial_{\Gamma\rtimes F_m}: K^\fin_0(C^\ast(\Gamma\rtimes F_m)) \to K_{1}(C_{L,0}^\ast(\widetilde X)^{\Gamma\rtimes F_m}) \]
is injective. In other words, $\partial_{\Gamma\rtimes F_m}$ maps a nonzero element in $ K^\fin_0(C^\ast(\Gamma\rtimes F_m))$ to a nonzero element in $K_{1}(C_{L,0}^\ast(\widetilde X)^{\Gamma\rtimes F_m}) $. 
\end{enumerate} 
To summarize, we have 
\begin{enumerate}[(a)]
\item $\sum_{k=1}^n c_k [p_{\gamma_k}] \neq 0$ in $K^\fin_0(C^\ast(\Gamma\rtimes F_m))$,
\item  $ \sum_{k=1}^n c_k \cdot i_\ast[\varrho(h_{\gamma_k})] = 0 $ in $K_{1}(C_{L,0}^\ast(\widetilde X)^{\Gamma\rtimes F_m})$, 
\item the map $ \partial_{\Gamma\rtimes F_m}: K^\fin_0(C^\ast(\Gamma\rtimes F_m)) \to K_{1}(C_{L,0}^\ast(\widetilde X)^{\Gamma\rtimes F_m})$ is injective,
\item and by Equation $\eqref{eq:bd}$,  $ \partial_{\Gamma\rtimes F_m} \Big( \sum_{k=1}^n c_k [p_{\gamma_k}] \Big) =  \sum_{k=1}^n c_k \cdot i_\ast[\varrho(h_{\gamma_k})]$.
\end{enumerate}
Therefore, we arrive at a contradiction. This finishes the proof.
\end{proof}

\section{Higher index, higher rho and positive scalar curvature at infinity}\label{sec:rhobdry}

In this section, we will first describe a construction of the higher index for the Dirac operator on 
a complete manifold with positive scalar curvature at infinity. This construction is due to Gromov-Lawson in the classic Fredholm case \cite{GL1} and its generalization to higher index case is due to Bunke \cite{B} (see also \cite{Roe1, BW, Roe3}). We will then discuss a connection between the higher index for the Dirac operator on a manifolds with boundary and the higher rho invariant of the Dirac operator on the boundary.

Let $M$ be a complete Riemannian  spin manifold with a proper and isometric action of a discrete group $\Gamma$.  We assume that $M$  has positive scalar curvature at infinity  relative to the action of $\Gamma$, i.e. there exists a $\Gamma$-cocompact subset $Z$ of $M$ and a positive number $a$ such that the scalar curvature $k$ of $M$ is greater than or equal to $ a$ outside $Z$. Let $D$ be the Dirac operator  $M$.

We need some preparations in order to define the higher index.
The following useful lemma is due to Roe \cite{Roe3}.

\begin{lem}\label{lm:estimate}  With the notation as above, suppose that $f\in S(\mathbb{R}) $ has its Fourier transform $\hat{f}$ supported in $(-r, r)$. Let $\phi \in C_0(M)$ have support disjoint from a $4r$-neighborhood of $Z$.
We have$$ || f(D) \phi || \leq 2 ||\phi || ~\mathrm{sup}\{|f(\lambda)|: |\lambda |\geq a\}. $$
\end{lem}
\begin{proof}
Let us first deal with the case where $f$ is an even function.
In the case, the Fourier transform formula gives us
$$f(D) =\int_{-r}^{r} \hat{f} (t) cos (tD) dt.$$
Let us define $$U= \{ x\in M: d(x, Z) > r\} \textup{ and } U'=  \{ x\in M: d(x, Z) >2 r\}.$$
Consider the unbounded symmetric operator $D^2$ with domain $C_c^\infty (U)$. This operator is bounded below by $a^2I$ and has a Friedrichs extension on the Hilbert space $L^2(U, S)$, where $S$ is the spinor bundle.
We denote this extension by $E$. Clearly, $E$ is bounded below by the same lower  bound $a^2I$. 

A standard finite propagation speed argument shows that if $s$ is smooth and compactly supported on $U'$, then 
$$cos(t D) s = cos (t\sqrt{E}) s$$
for $-r \leq t \leq r.$
Since the spectrum of $\sqrt{E}$ is bounded below by $a$, we have
$$||f(\sqrt{E})|| \leq  ~\mathrm{sup}\{|f(\lambda)|: |\lambda |\geq a\}. $$
This implies the following  inequality:
$$ || f(D) \phi || \leq ||\phi || ~\mathrm{sup}\{|f(\lambda)|: |\lambda |\geq a\}.$$

If $f$ is an odd function, we have 
$$  || f(D) \phi ||^2 \leq ||\bar{\phi} || ~~~|| ~f^2 (D) \phi ~||.       $$
In this case, the function $f^2$ is even, belongs to $S(\mathbb R)$, and has Fourier transform supported in $(-2r, 2r)$. Hence  we have the following inequality:
$$ || f^2(D) \phi || \leq ||\phi || ~\mathrm{sup}\{|f(\lambda)|^2: |\lambda |\geq a\}.$$ It follows that 
$$ || f(D) \phi || \leq ||\phi || ~\mathrm{sup}\{|f(\lambda)|: |\lambda |\geq a\}.$$

The general case follows from the above two special cases by writing $f$ as a sum of even and odd functions.
\end{proof}

With the help of the above lemma, we can prove the following result.

\begin{lem} For any $f\in C_c(-a, a)$, we have $f(D) \in \lim_{R \rightarrow \infty} C^\ast( N_R (Z))^\Gamma$, where $N_R (Z) $ is 
the $R$-neighborhood\footnote{Without loss of generality, we can assume $N_R(Z)$ is $\Gamma$-invariant.} of $Z$ and $\lim_{R \rightarrow \infty} C^\ast( N_R (Z))^\Gamma$  is the $C^\ast$-algebra limit of the equivariant Roe algebras.
\end{lem}

\begin{proof}
 For any $\epsilon>0$, there exists a smooth function $g$ such that
its Fourier transform is compactly supported, and 
$$\mathrm{sup} \{ |g(\lambda)-f(\lambda)| : \lambda\in \mathbb{R}\}< \epsilon
.$$
It follows that $|g(\lambda)| < \epsilon$ for $|\lambda |>a.$ Let $r$ be a positive number such that $ \mathrm{Supp} (\hat{g})\subseteq (-r, r)$ and let $\psi: M \rightarrow [0,1]$ be a continuous $\Gamma$-invariant  function equal to $1$ on a $4r$-neighborhood of $Z$ and vanishing outside a $5r$-neighborhood of $Z$. We write
$$f(D)= \psi g (D) \psi + (1-\psi) g (D) \psi + g (D) (1-\psi) +(f (D)-g (D)) .$$
Note that the first term is a $\Gamma$-equivariant and locally compact  operator 
with finite propagation supported near $Z$, the second and third terms 
have norm bounded by $2\epsilon$ by Lemma $\ref{lm:estimate}$. This implies the desired result.
\end{proof}

We remark  $\lim_{R \rightarrow \infty} C^\ast( N_R (Z))^\Gamma$  
is isomorphic to the reduced group $C^\ast$-algebra $C^\ast_r(\Gamma)$.

A normalizing function $\chi\colon \mathbb R\to [-1, 1]$ is, by definition, a continuous odd function that goes to $\pm 1$ as $x\to \infty$.  Now choose a normalizing function $\chi$ such that $\chi^2-1$ is supported in $(-a, a)$ and let $$ F= \chi(D).$$
By Lemma 7.2,  the same construction from Section $\ref{sec:index}$ defines a higher index $Ind(D) \in K_\ast (C^\ast_r(\Gamma)).$

The following question is wide open.

\begin{open} Let $M$ be a complete spin manifold with a proper and isometric action of a discrete group $\Gamma$. Let $D$ be the Dirac operator on $M$. Assume that $M$ has positive scalar curvature at infinity relative to the action of $\Gamma$. Is $Ind(D)$ an element in the image of the Baum-Connes assembly map?

\end{open}

 Let $N$ be a spin manifold with boundary, where the boundary $\partial N$ is endowed with a positive scalar curvature metric. We will explain that the $K$-theoretic ``boundary'' of the higher index class of the Dirac operator on $N$ is equal  to the higher rho invariant of the Dirac operator on  $\partial N$.  More generally, let $M$ be an $m$-dimensional complete spin manifold with boundary $\partial M$ such that
\begin{enumerate}[(i)]
\item the metric on $M$ has product structure near $\partial M$ and its restriction on $\partial M$, denoted by $h$, has positive scalar curvature;
\item there is a proper and cocompact isometric action of a discrete group $\Gamma$ on $M$;
\item the action of $\Gamma$ preserves the spin structure of $M$. 
\end{enumerate}  
We denote the associated Dirac operator on $M$ by $D_{M}$ and the associated Dirac operator on $\partial M$ by $D_{\partial M}$. With the positive scalar curvature metric $h$ on the boundary $\partial M$, we can define the higher index class $Ind(D_M)$ of $D_M$  in $K_\ast(C^\ast_r (\Gamma))$ as follows. We can attach a cylinder $\partial M \times [0,\infty)$ to the boundary of $M$ to form a complete Riemannian manifold (without boundary) $\bar M$, where the Riemannian metric on $M$ is naturally extended to $\bar M$ such that Riemannian metric on the cylinder is a product. The action of $\Gamma$ on $M$ naturally extends to an action on $\bar M$. By construction,  $\bar M$ has positive scalar curvature at infinity relative to the action of $M$. We can therefore define the higher index   $Ind(D_M)$ of $D_M$ to be the higher index of the Dirac operator on $\bar M$.

 Notice that the short exact sequence of $C^\ast$-algebras
\[  0\to C^\ast_{L,0}(M)^\Gamma \to C^\ast_L(M)^\Gamma \to C^\ast(M)^\Gamma \to 0\]
induces the following long exact sequence in $K$-theory:
\[ \cdots \to K_i(C_L^\ast(M)^\Gamma) \to  K_i(C^\ast(M)^\Gamma) \xrightarrow{\partial_i } K_{i-1}(C^\ast_{L,0}(M)^\Gamma) \to  K_{i-1}(C^\ast_L(M)^\Gamma) \to \cdots. \]
Also, by functoriality, we have a natural homomorphism 
\[ \iota_\ast: K_{m-1}(C^\ast_{L,0}(\partial M)^\Gamma)  \to K_{m-1}(C^\ast_{L,0}(M)^\Gamma) \]
induced by the inclusion map $\iota: \partial M \hookrightarrow M$. With the above notation, one has the following theorem. 
\begin{thm} \label{thm:bdry}
	$\partial_m(Ind(D_{M}) )  = \iota_\ast( \rho(D_{\partial M}, h))$
in $K_{m-1}(C^\ast_{L,0}(M)^\Gamma)$. 
\end{thm} 

 This theorem is due to Piazza and Schick \cite{PS1} when the dimension of $M$ is even and to Xie and Yu \cite{XY} in the general case.  As an immediate application, one sees that nonvanishing of the higher rho invariant is an obstruction to extension of the positive scalar curvature metric from the boundary to the whole manifold. Moreover, the higher rho invariant can be used to distinguish whether or not two positive scalar curvature metrics are connected by a path of positive scalar curvature metrics. 
 

\section{Secondary invariants of the signature operators and topological non-rigidity}

In this section, we introduce the higher rho invariants for a pair of closed manifolds which are homotopic equivalent to each other. Roughly speaking, we consider the relative signature operator associated to this pair of manifolds. 
This relative signature operator has trivial higher index with a natural trivilialization given  by the homotopy equivalence.  This trivialization allows us to define a higher rho invariant, which can be used to detect whether a homotopy equivalence can be deformed into a homeomorphism.

We shall focus on the case of smooth manifolds. General topological manifolds can be handled in a similar way with the help of Lipschitz structures \cite{Su}.

Let $M$ and $N$ be two closed oriented smooth manifolds of dimension $n$. We will only discuss the odd dimensional case; the even dimensional case is completely similar.  
We denote the de Rham complex of differential forms on $M$ by 
\[   \Omega^0(M) \xrightarrow{\ d\ } \Omega^1(M) \xrightarrow{\ d \ } \cdots  \xrightarrow{\ d\ } \Omega^n(M), \]
whose $L^2$-completion is
\[   \Omega^0_{L^2}(M) \xrightarrow{\ d\ } \Omega^1_{L^2}(M) \xrightarrow{\ d \ } \cdots  \xrightarrow{\ d\ } \Omega^n_{L^2}(M).\]
We shall write $d_M$ if we need to specify $d$ is the differential associated to the de Rham complex of $M$. 
Similarly, we have 
\[   \Omega^0_{L^2}(N) \xrightarrow{\ d_N\ } \Omega^1_{L^2}(N) \xrightarrow{\ d_N \ } \cdots  \xrightarrow{\ d_n\ } \Omega^n_{L^2}(N)\]
for the manifold $N$.

Let $T = \ast_M\colon \Omega^k_{L^2}(M) \to \Omega^{n-k}_{L^2}(M)$ be the Hodge star operator on $M$, which is defined by 
\[ \langle T\alpha, \beta\rangle  = \int_M \alpha \wedge \overline \beta, \]
where $\overline \beta$ is the complex conjugation of $\beta$.
The Hodge star operator $T$ satisfies the following properties:
\begin{enumerate}[(1)]
	\item $T^\ast \alpha = (-1)^{k(n-k)} T \alpha$ for any $\alpha \in \Omega^k_{L^2}(M)$;
	\item $Td\alpha + (-1)^k d^\ast T\alpha = 0$ for any smooth $\beta\in \Omega^k(M)$; 
    \item $T^2 \alpha = (-1)^{nk+k}\alpha$ for any $\alpha \in \Omega^k_{L^2}(M)$;
\end{enumerate}
where $T^\ast$ is the adjoint of $T$, and $d^\ast$ is the adjoint of $d$. More generally, a bounded operator $T$ satisfying conditions $(1)$ and $(2)$ is said to be a \emph{duality operator} of the chain complex $(\Omega^\ast_{L^2}(M), d)$ if
in addition, it satisfies the condition
\begin{enumerate}[$(3)'$]
\item $T$ induces a chain homotopy equivalence from the dual complex of $(\Omega^\ast_{L^2}(M), d)$ to the complex $(\Omega^\ast_{L^2}(M), d)$, where the dual complex is defined to be
\[   \Omega^n_{L^2}(M) \xrightarrow{\ d^\ast\ } \Omega^{n-1}_{L^2}(M) \xrightarrow{\ d^\ast \ } \cdots  \xrightarrow{\ d^\ast\ } \Omega^0_{L^2}(M).\]
\end{enumerate}
In this case, we call $(\Omega^\ast_{L^2}(M), d)$ together with the duality operator $T$ a (unbounded) Hilbert-Poincar\'e complex.

Define $S = i^{k(k-1)+m}T$, where $m = (n-1)/2$. It follows from properties $(1)$ and $(3)$ above that $S$ is a selfadjoint involution.  

\begin{defns}
	The signature operator $D$ of $M$ is defined to be $i(d+d^\ast) S$ acting on even degree differential forms. 
\end{defns}

All the above discussion generalizes to the universal covering $\widetilde M$ of $M$. We denote the corresponding $\pi_1(M)$-equivariant signature operator of $\widetilde M$ by $\widetilde D$. 

In the standard $K$-theoretic construction of the index of $\widetilde D$ (cf. Section $\ref{sec:index}$), let us choose the normalizing function $\chi(t) = \frac{2}{\pi}\arctan(t)$. In this case, we have  
\[  Ind(\widetilde D) = e^{2\pi i \frac{\chi(\widetilde D) +1}{2} } = (\widetilde D - i) (\widetilde D+i)^{-1}. \]
Let $B=d+d^\ast$.
The above formula implies the following index formula:
\begin{equation}\label{eq:sigind}
Ind(\widetilde D) = (B-S)(B+S)^{-1} \in K_1(C^\ast_r(\pi_1(M))).
\end{equation}

The above index formula in fact holds for general Hilbert-Poincar\'e complexes, that is,  chain complexes with general duality operators.  We shall not get into the technical details regarding the notion of Hilbert-Poincar\'e complexes, but instead refer the reader to \cite{HR} for details. A key feature of the notion of Hilbert-Poincar\'e complexes is that it allows us to use a much larger class of duality operators besides the Hodge star operators. 
In the case of general Hilbert-Poincar\'e complexes, the well-definedness of the above index formula  is justified by the following 
 lemma \cite[Lemma 3.5]{HR}.
\begin{lem} $B+S$ and $B-S$ are invertible.
\end{lem}

\begin{proof}
	 Consider the mapping cone complex associated to the chain map 
$$S: (\Omega^{\ast}_{L^2} (M), -d^\ast) \rightarrow (\Omega^\ast_{L^2} (M), d)$$ with the differential 
\[ b=\begin{pmatrix}
d &  0\\
S & d^\ast
\end{pmatrix}. \]
Since $S$ is an isomorphism on the homology,  the mapping cone complex is acyclic.  Therefore the operator $b+b^\ast$ is invertible.
Recall that $S$ is self-adjoint. Hence we have 
\[ b+b^\ast= \begin{pmatrix}
d+d^\ast&  S\\
S & d+d^\ast
\end{pmatrix}.\]
Note that 
\[\begin{pmatrix}
d+d^\ast&  S\\
S & d+d^\ast
\end{pmatrix}   \begin{pmatrix}
v\\
v
\end{pmatrix} = \begin{pmatrix}
(d+d^\ast+S)v\\
(d+d^\ast+S)v
\end{pmatrix}.\]
This implies that $B+S$ is invertible. We can similarly show that
$B-S$ is invertible. 
\end{proof}


Suppose $f\colon M\to N$ is an orientation preserving homotopy equivalence between $M$ and $N$. It is known that $Ind(\widetilde D_M)  = Ind(\widetilde D_N)$ in $K_1(C^\ast_r(\Gamma))$, where $\Gamma = \pi_1(M) = \pi_1(N)$, cf.  \cite{K1}\cite{KM}. Intuitively speaking, one can use the   homotopy equivalence $f$ together with the signature operators on $M$ and $N$ to produce an invertible operator $D_f$ on $M\cup (-N)$ such that the index of $D_f$ coincides with the index of the signature operator on $M\cup (-N)$, which  is $Ind(\widetilde D_M) -Ind(\widetilde D_N)$, cf. \cite{HiS}. Here $-N$ is the manifold $N$ with the reversed orientation and $M\cup (-N)$ stands for the disjoint union of the two. In particular, the invertibility of $D_f$ is the reason that $Ind(\widetilde D_M) = Ind(\widetilde D_N)$, by giving a specific trivialization of the index class of $D_{M\cup (-N)}$. Thus the homotopy equivalence $f$ naturally defines a higher rho invariant. In the following, we shall take a  different but simpler approach to construct the higher rho invariant of $f$. Although this process does not produce an invertible operator $D_f$, but it does provide a  trivialization at the K-theory level. Our choice of such an approach is mainly its simplicity, which will hopefully convey the key ideas with more clarity.

We denote the induced pullback map on differential forms by $f^\ast\colon \Omega^\ast(N) \to \Omega^\ast(M)$. In general, $f^\ast$ does not extend to a bounded linear map between the spaces of $L^2$ forms $\Omega_{L^2}^\ast(N)$ and $\Omega_{L^2}^\ast(M)$. In order to fix this issue, we need the following construction due to Hilsum and Skandalis \cite{HiS}. First, suppose $\varphi: X\to Y$ is a submersion between two closed manifolds. It is easy to see that $\varphi^\ast$ does extend to a bounded linear operator from $\Omega_{L^2}^\ast(Y)$ to $\Omega_{L^2}^\ast(X)$. Now let $\iota\colon N\to \mathbb R^k$ be an embedding. Suppose $U$ is a tubular neighborhood of $N$ in $\mathbb R^k$ and $\pi\colon U \to N$ is the associated projection. Without loss of generality, we assume $\iota(N) + \mathbb B^k \subset U$, where $\mathbb B^k$ is the unit ball of $\mathbb R^k$. Let $p\colon M\times \mathbb B^k\to N$ be the submersion defined by $p(x, t) = \pi(f(x) + t)$. Furthermore, let $\omega$ be a volume form on $\mathbb B^k$ whose integral is $1$. Then the formula
\[   \alpha \to \int_{\mathbb B^k}p^\ast(\alpha)\wedge \omega  \] defines a morphism of chain complexes $A \colon \Omega^\ast(N) \to \Omega^\ast(M)$, where $\int_{\mathbb B^k}$ denotes fiberwise integration along $\mathbb B^k$. It is easy to see that $A$ extends to a bounded linear operator from $\Omega^\ast_{L^2}(N)$ to $ \Omega^\ast_{L^2}(M)$. We shall still denote this extension by $A\colon \Omega^\ast_{L^2}(N) \to \Omega^\ast_{L^2}(M)$.

Now a routine calculation shows that $A$ is a homotopy equivalence between the two complexes $(\Omega_{L^2}(M), d_M)$ and $(\Omega_{L^2}(N), d_N)$ such that $ A TA^\ast$ is chain homotopy equivalent to $T'$, where $T'$ is the Hodge star operator on $N$. It follows that the operator 
\[S= \begin{pmatrix}
0 &  A T\\
TA^\ast & 0
\end{pmatrix} \]
together with the chain complex $(\Omega_{L^2}^\ast(M) \oplus \Omega_{L^2}^\ast(N), d_M \oplus d_N)$ gives rise to an (unbounded) Hilbert-Poincar\'e complex. 

We have the following lemma due to Higson and Roe \cite{HR}.

\begin{lem} If we write  \[ B = \begin{pmatrix}
 B_M  &  0 \\ 0 & B_N
 \end{pmatrix} = \begin{pmatrix}
 d_M + d_M^\ast & \\
 0 & d_N+ d_N^\ast
 \end{pmatrix}, \] 
then the element 
\[  (B -S) (B + S)^{-1} \]
is equal to $Ind(\widetilde D_M) - Ind(\widetilde D_N)$ in $K_1(C^\ast_r(\Gamma))$.
\end{lem}
\begin{proof} 
Note that $T'$ and $ATA^\ast$ induce the same map on homology.
It follows that the path
\[ \begin{pmatrix}
T &  0\\
0 &  (s-1)T'-sATA^\ast
\end{pmatrix} \] 
is an operator homotopy connecting the duality operator $T\oplus -T'$ to the duality operator  
$T\oplus -ATA^\ast$. The path
\[ \begin{pmatrix}
cos (s) T &  sin(s) TA^\ast\\
sin(s) AT & -cos(s) ATA^\ast
\end{pmatrix} \] 
is an operator homotopy connecting the duality operator $T\oplus -ATA^\ast$ to the duality operator $ \begin{pmatrix}
0 &  A T\\
TA^\ast & 0
\end{pmatrix},$ where $s\in [0, \pi/2].$
Now the lemma follows from the  explicit index formula in line $\eqref{eq:sigind}$. 
\end{proof}

 For each $t\in [0, \pi]$, the following operator 
\[ S_t= \begin{pmatrix}
0 &  e^{it}A T\\
e^{-it} TA^\ast & 0
\end{pmatrix} \]
defines a duality operator for the chain complex $(\Omega_{L^2}^\ast(M) \oplus \Omega_{L^2}^\ast(N), d_M \oplus d_N)$. 
  It is not difficult to verify that 
\[  (B-S_0)(B+S_t)^{-1} \]
defines a continuous path of invertible elements in $(C^\ast_r(\Gamma)\otimes \mathcal K)^+$. Note that $S_\pi = -S_0$, thus $(B-S_0)(B+S_\pi)^{-1} = 1$. Therefore, the path $(B-S_0)(B+S_t)^{-1}$ gives a specific trivialization of the index class $Ind(\widetilde D_M) - Ind(\widetilde D_N)$. This trivialization in  turn  induces a higher rho invariant as follows. Let $\{v(t)\}_{1\leq t \leq 2}$ be the path of invertible elements connecting $(B-S_0)(B+S_{0})^{-1}$ to 
\[ \big( B - \begin{psmallmatrix}
T & 0  \\ 0 & T'
\end{psmallmatrix}\big) \big( B  + \begin{psmallmatrix}
T & 0  \\ 0 & T'
\end{psmallmatrix}\big)^{-1}.  \]  We define 
\[ \rho(t)= \begin{cases}
(B-S_0)(B+S_{(1-t)\pi})^{-1},  &  \textup{ for } 0\leq t\leq 1,  \\
\\
v(t) & \textup{ for } 1\leq t \leq 2,
\\
    \\
\begin{pmatrix}
e^{\pi i (\chi(\frac{1}{t}\widetilde D_M)+1)} &  0\\
0 & e^{\pi i (\chi(\frac{1}{t}\widetilde D_{-N})+1)}
\end{pmatrix} & \textup{ for }  2\leq t<\infty.
\end{cases} \]

\begin{defns}
We define the higher rho invariant of a given homotopy equivalence 
$f\colon  M \rightarrow N$ to be the above element $[\rho]$ in  $K_1(C_{L,0}^\ast(\widetilde{N})^\Gamma)$. Here we have used $f$ to map elements in
$C^\ast(\widetilde{M})^\Gamma$ to $C^\ast(\widetilde{N})^\Gamma$.
\end{defns}

The fact that $\rho$ is an element in the matrix algebra of 
$(C_{L,0}^\ast(\widetilde{N})^\Gamma)^+$  follows from a standard finite propagation speed argument. 
The even dimensional higher rho invariant can  be defined in a similar way.
Zenobi generalized the concept of higher rho invariant to homotopy equivalences
between closed topological manifolds with the help of Lipschitz structures \cite{Z}.

Given a closed oriented manifold $N$, the higher rho invariant in fact defines a map from the structure set of $N$ to $K_n(C_{L,0}^\ast(\widetilde{N})^\Gamma)$, where $n=\dim N$. On the other hand, when $N$ is a \emph{topological} manifold, the structure set of $N$ carries a natural abelian group structure. It was long standing open problem whether the higher rho inviariant map is a group homomorphism from the structure group of $N$ to $K_1(C_{L,0}^\ast(\tilde{N})^\Gamma)$. This was answered in positive in complete generality by Weinberger-Xie-Yu \cite{WXY}. In the following we shall briefly discuss some of the key ideas of their proof and also some applications to topology. 

   Let $X$ be a closed oriented connected \emph{topological} manifold of dimension $n$. The structure group $\mathcal S(X)$ is the abelian group of equivalence classes of all pairs $(f, M)$ such that $M$ is a closed oriented manifold and $f\colon M \to X$ is an orientation-preserving homotopy equivalence. Recall that the abelian group structure on $\mathcal S(X)$ is originally described through the Siebenmann periodicity map, which is an injection from $\mathcal S(X)$ to $\mathcal S_\partial(X\times D^4)$, where $D^4$ is the $4$-dimensional Euclidean unit ball and $\mathcal S_\partial(X\times D^4)$ is the rel$\,\partial$ version of structure set of $X\times D^4$.  The set $\mathcal S_\partial(X\times D^4)$  carries a natural abelian group structure by stacking, and induces an abelian group structure on $\mathcal S(X)$ by Nicas' correction map to the Siebenmann periodicity map \cite{Ni}. Both $\mathcal S(X)$ and   $\mathcal S_\partial(X\times D^4)$ carry a higher rho invariant map. It is not difficult to verify that the higher rho invariant map on $\mathcal S_\partial(X\times D^4)$ is additive, i.e. a homomorphism between abelian groups. One possible approach to show the additivity of the higher rho invariant map on $\mathcal S(X)$ is to prove  the compatibility of higher rho invariant maps on $\mathcal S(X)$ and   $\mathcal S_\partial(X\times D^4)$. However, there are some essential analytical difficulties to \emph{directly} prove such a compatibility,  due to the subtleties of the Siebenmann periodicity map\footnote{A geometric construction of the Siebenmann periodicity map was given by Cappell and Weinberger \cite{CW}}.   A main novelty of Weinberger-Xie-Yu's approach \cite{WXY} is to give a new description  of the topological structure group in terms of smooth manifolds with boundary. This new description uses more objects and an equivalence relation broader than $h$-cobordism, which allows us to replace topological manifolds in the usual definition of $\mathcal S(X)$ by smooth manifolds with boundary. Such a description   leads to a transparent group structure, which is given by disjoint union. The main body of Weinberger-Xie-Yu's work \cite{WXY} is devoted to proving that the new description coincides with the classical description of the topological structure group; and to developing the theory of higher rho invariants in this new setting, in which higher rho invariants are easily seen to be additive.    As a consequence, the higher rho invariant maps on   $\mathcal S(X)$ and   $\mathcal S_\partial(X\times D^4)$ are indeed compatible. 


\begin{thm}[{\cite[Theorem 4.40]{WXY}}] \label{thm:add} The higher rho invariant map is a group homomorphism from 
$\mathcal S(X)$  to $K_n(C_{L,0}^\ast(\widetilde{X})^\Gamma).$

\end{thm}

As mentioned above, the above theorem solves the long standing open problem  whether the higher rho inviariant map defines a group homomorphism on the topological structure group.
As an application,  Weinberger-Xie-Yu applied the above theorem to prove that the structure groups of certain manifolds are infinitely generated \cite{WXY}. 

\begin{thm} Let $M$ be a closed oriented topological manifold of dimension $n \geq 5$, and $\Gamma$ be its fundamental group. Suppose the rational strong Novikov conjecture holds for $\Gamma$. If $\oplus_{k\in \mathbb{Z}} H_{n+1+4k} (\Gamma, \mathbb{C})$
is infinitely generated, then the topological structure group of $S(M)$
is infinitely generated.
\end{thm}

We refer to the article \cite{WXY} for examples of groups satisfying the conditions in the above theorem. 


\section{Non-rigidity of topological manifolds and reduced structure groups}\label{sec:nonrig}

The structure group measures the degree of non-rigidity and the reduced structure group estimates the size of non-rigidity modulo self-homotopy equivalences.
In this section, we apply the higher rho invariants of signature operators to give a lower bound of the free rank of reduced structure groups of closed oriented topological manifolds. Our key tool is the additivity property of higher rho invariants from the previous section. There are in fact  two different versions of reduced structure groups, $\widetilde{\mathcal S}_{alg}(X)$ and $\widetilde {\mathcal S}_{geom}(X)$, whose precise definitions will be given below.  The group $\widetilde{\mathcal S}_{alg}(X)$ is functorial and fits well with the surgery long exact sequence. On the other hand, the group $\widetilde {\mathcal S}_{geom}(X)$ is more geometric in the sense that it measures  the size of the collection of closed manifolds homotopic equivalent but not homeomorphic to $X$.   

Since we will be using the maximal version of various $C^\ast$-algebras throughout this section, we will omit the subscript ``max'' for notational simplicity. 

Let $X$ be an $n$-dimensional oriented closed topological manifold. Denote the monoid of orientation-preserving self homotopy equivalences of $X$ by $Aut_{h}(X)$. There are two different actions of $Aut_{h}(X)$ on $\mathcal S(X)$, which induce two different versions of reduced structure groups as follows.

On one hand,  $Aut_{h}(X)$ acts naturally on ${\mathcal S}(X)$ by 
\[   \alpha_u(\theta) =  u_\ast (\theta) \] 
for all $u\in Aut_{h}(X)$ and all $\theta\in {\mathcal S}(X)$, where 
$u_\ast$ is the group  homomorphism from ${\mathcal S}(X)$ to  ${\mathcal S}(X)$ induced by the map $u$ \cite{KiS}.
This action $\alpha$ is compatible with the actions of $Aut_{h}(X)$ on other terms in the topological surgery  exact sequence.

On the other hand,  $Aut_{h}(X)$ also naturally acts on ${\mathcal S}(X)$ by compositions of homotopy equivalences, that is,    
\[   \beta_u(\theta) = (u\circ f, M)    \] 
for all $u\in Aut_{h}(X)$ and all $\theta =  (f, M) \in {\mathcal S}(X)$. 
Note that 
\[ \beta_u \colon {\mathcal S}(X) \to {\mathcal S}(X) \] only defines a bijection of sets, and is not a group homomorphism in general. 
\begin{defns}\label{def:rdstr} With the same notation as above, we define the following reduced structure groups. 
	\begin{enumerate}[(1)]
		\item 	Define 	$\widetilde{\mathcal S}_{alg}(X)$  to be the quotient group of ${\mathcal S}(X)$ by the subgroup generated by elements of the form $\theta - \alpha_u (\theta)$ for all $\theta\in \mathcal S(X)$ and all $u\in Aut_{h}(X)$. 
		\item 	we define 	$\widetilde{\mathcal S}_{geom}(X)$ to be the quotient group of ${\mathcal S}(X)$  by the subgroup generated by elements of the form $\theta - \beta_u(\theta)$ for all $\theta\in \mathcal S(X)$ and all $u\in Aut_{h}(X)$. 
	\end{enumerate}
	
\end{defns}

 
 Next we recall a method of constructing elements in the structure group by the finite part of $K$-theory \cite[Theorem 3.4]{WY}. 

	Let $M$ be a $(4k-1)$-dimensional closed oriented connected topological manifold with $\pi_1 M = \Gamma$. Suppose $\{g_1, \cdots, g_m\}$ is a collection of elements in $\Gamma$ with distinct finite orders  such that $g_i\neq e$ for all $1\leq i\leq m$.  Recall  the topological surgery exact sequence:
	\[ \cdots \rightarrow H_{4k}(M, \mathbb{L}_\bullet) \rightarrow L_{4k}(\Gamma) \xrightarrow{\mathscr S} \mathcal S(M) \rightarrow H_{4k-1}(M, \mathbb{L}_\bullet )\rightarrow \cdots. \]
	For each finite subgroup $H$ of $\Gamma$, we have the following commutative diagram:
	\[ 
	\xymatrix{ H_{4k}^{H}(\underline{E}H , {\mathbb L_\bullet}) \ar[r]^-A  \ar[d] &  L_{4k}(H) \ar[d] \\
		H_{4k}^{G}(\underline{E}\Gamma , {\mathbb L}_\bullet) \ar[r]^-A   &  L_{4k}(\Gamma),	}
	\]
	where the vertical maps are induced by the inclusion homomorphism from $H$ to $\Gamma$.
	For each element $g$ in $H$ with finite order $d$, the idempotent  $p_g = \frac{1}{d}(\sum_{k=1}^{d} g^k)$ produces a class in $L_0 ( \mathbb{Q} H)$, where $L_0 ( \mathbb{Q} H)$ is the algebraic definition of $L$-groups using quadratic forms and formations with coefficients in $\mathbb Q$.  Let
	$[q_g]$ be the corresponding element in $L_{4k} ( \mathbb{Q} H)$ given by periodicity.
	Recall that $$L_{4k} (H) \otimes \mathbb{Q} \simeq L_{4k} (\mathbb{Q} H) \otimes \mathbb{Q}.$$
	For each element $g$ in $H$ with finite order, we use the same notation $[q_g]$ to denote  the element in $L_{4k}(H)\otimes \mathbb{Q}$ corresponding to   $[q_g]\in L_{4k} ( \mathbb{Q} H)$ under the above isomorphism.
	
	We also have the following commutative diagram:
	\[ 
	\xymatrix{ H_{4k}^{\Gamma}(E\Gamma , {\mathbb L_\bullet}) \otimes \mathbb Q \ar[r]^-A  \ar[d] &  L_{4k}(\Gamma) \otimes \mathbb Q \ar[d] \\
		K_0^\Gamma(E\Gamma)\otimes \mathbb Q \ar[r]^-{\mu_\ast}   &  K_{0}( C^\ast(\Gamma))\otimes \mathbb Q,	}
	\]
	where the left vertical map is induced by a map at the spectra level and the right vertical map is induced by the inclusion map:
	$$ L_{4k}(\Gamma)\rightarrow L_{4k}( C^\ast (\Gamma) ) \cong  K_0( C^\ast (\Gamma))$$ (see \cite{R2} for the last identification).
	
	Now if $\Gamma$ is finitely embeddable into Hilbert, then the abelian subgroup of $K_0(C^\ast(\Gamma))$ generated by $\{[p_{g_1}], \cdots, [p_{g_m}]\}$ is not in the image of
	of the map 
	\[ \mu_\ast:
	K_0^\Gamma(E\Gamma)\to K_0(C^\ast(\Gamma)). \] It follows that 
	\begin{enumerate}[(1)]
		\item any nonzero element in the abelian subgroup of $L_{4k} (  \Gamma)\otimes \mathbb{Q}$ generated by the elements $\{[q_{g_1}], \cdots, [q_{g_m}]\}$
		is not in the image of the rational assembly map \[ A: H_{4k}^{\Gamma}(E\Gamma , \mathbb L_\bullet) \otimes \mathbb{Q} \to L_{4k} ( \Gamma)\otimes \mathbb{Q}; \]
		\item the abelian subgroup of $L_{4k} (\Gamma)\otimes \mathbb{Q}$ generated by $\{[q_{g_1}], \cdots, [q_{g_m}]\}$ has rank $m$. 
	\end{enumerate}
	
	By exactness of the surgery sequence,
	we know that the map   
	\begin{equation}\label{eq:struc}
\mathscr S\colon L_{4k} ( \Gamma)\otimes \mathbb{Q}\to \mathcal S(M)\otimes \mathbb{Q},
	\end{equation} 
	is injective on the abelian subgroup of  $L_{4k} (\Gamma)\otimes \mathbb{Q}$ generated by $\{[q_{g_1}], \cdots, [q_{g_n}]\}$.

In order to prove the main result of this section, we need to apply the above argument not only to $\Gamma$,  but also to certain semi-direct products of $\Gamma$ with  free groups of finitely many generators.

Recall that $N_\fin(\Gamma)$ is the cardinality of the following collection of positive integers: \[ \{ d\in \mathbb N_+\mid \exists \gamma\in \Gamma \ \textup{such that} \ \gamma\neq e \ \textup{and}\ order(\gamma) = d \}. \] 
We have the following result \cite{WXY}.  At the moment, we are only able to  prove the theorem for $\widetilde{\mathcal S}_{alg}(M)$. We will give a brief discussion to indicate the difficulties in  proving the version  $\widetilde{\mathcal S}_{geom}(M)$ after the theorem. 
\begin{thm}\label{thm:structure}
	Let $M$ be a closed oriented topological manifold with dimension $ n = 4k-1$ \textup{($k>1$)} and $\pi_1 M = \Gamma$. If $\Gamma$ is strongly finitely embeddable into Hilbert space \textup{(}cf. Definition $\ref{def:strfe}$\textup{)}, then the free rank of  $\widetilde{\mathcal S}_{alg}(M)$ is  $\geq N_\fin(\Gamma)$. 
\end{thm}

\begin{proof}
 A key point of the argument below is to use a semi-direct product $\Gamma\rtimes F_m$ to turn certain outer automorphisms of $\Gamma$ into inner automorphisms of $\Gamma\rtimes F_m$. 

Consider the higher rho invariant homomorphism from Theorem $\ref{thm:add}$: 
\[  \rho: \mathcal S(M) \to  K_{1}(C_{L,0}^\ast(\widetilde M)^\Gamma). \]
Note that  every self-homotopy equivalence $\psi \in Aut_{h}(M)$ induces a homomorphism\footnote{Let us review how the homomorphism $\psi_\ast: K_{1}(C_{L,0}^\ast(\widetilde M)^\Gamma)\to K_{1}(C_{L,0}^\ast(\widetilde M)^\Gamma)$ is defined. The map $\psi\colon M\to M$ lifts to a map $\widetilde \psi\colon \widetilde M \to \widetilde M$. However, to view  $\widetilde \psi$ as a $\Gamma$-equivariant map, we need to use two different actions of $\Gamma$ on $\widetilde M$. Let $\tau$ be a right action of $\Gamma$ on $\widetilde M$ through deck transformations. Then we define a new action $\tau'$ of $\Gamma$ on $\widetilde M$ by $\tau'_g = \tau_{\psi_\ast(g)}$, where $\psi_\ast\colon \Gamma \to \Gamma$ is the automorphism induced by $\psi$. It is easy to see that $\widetilde \psi\colon \widetilde M \to \widetilde M$ is $\Gamma$-equivariant, when $\Gamma$ acts on the first copy of $\widetilde M$ by $\tau$ and the second copy of $\widetilde M$ by $\tau'$. Let us denote the corresponding $C^\ast$-algebras by $C_{L,0}^\ast(\widetilde M)_{\tau}^\Gamma$ and $C_{L,0}^\ast(\widetilde M)_{\tau'}^\Gamma$. Observe that,  despite  the two different actions of $\Gamma$ on $\widetilde M$,  the two $C^\ast$-algebras $C_{L,0}^\ast(\widetilde M)_{\tau}^\Gamma$ and $ C_{L,0}^\ast(\widetilde M)_{\tau'}^\Gamma$ are canonically identical, since an operator is invariant under the action $\tau$ if and only if it is invariant under the action $\tau'$. 	 }
\[\widetilde \psi_\ast: K_{1}(C_{L,0}^\ast(\widetilde M)^\Gamma)\to K_{1}(C_{L,0}^\ast(\widetilde M)^\Gamma).\]
Let $\mathcal I_1(C_{L,0}^\ast(\widetilde M)^\Gamma)$ be the subgroup of $K_{1}(C_{L,0}^\ast(\widetilde M)^\Gamma)$  generated by elements of the form $[x] - \widetilde \psi_\ast[x]$ for all $[x]\in K_{1}(C_{L,0}^\ast(\widetilde M)^\Gamma)$ and all $\psi\in Aut_{h}(M)$. Note that, by the definition of the higher rho invariant, we have
\[  \rho(\alpha_\psi(\theta)) = \widetilde \psi_\ast(\rho(\theta))  \in K_1(C_{L,0}^\ast(\widetilde M)^\Gamma) \]
for all $\theta\in \mathcal S(M)$ and $\psi\in Aut_{h}(M)$. It follows that  $\rho$ descends to a group homomorphism
$\widetilde{\mathcal S}_{alg}(M) \to K_{1}(C_{L,0}^\ast(\widetilde M)^\Gamma)\big/\mathcal I_1(C_{L,0}^\ast(\widetilde M)^\Gamma).$

Now for a collection of elements $\{\gamma_1, \cdots, \gamma_\ell\}$ with distinct finite orders, we consider the elements $\mathscr S(p_{\gamma_1}), \cdots, \mathscr S(p_{\gamma_\ell}) \in \mathcal S(M)$ as in line $\eqref{eq:struc}$. To be precise, the elements $\mathscr S(p_{\gamma_1}), \cdots, \mathscr S(p_{\gamma_\ell})$ actually lie in  $\mathcal S(M) \otimes \mathbb Q$. Consequently, all abelian groups in the following need to be tensored by the rationals $\mathbb Q$. For simplicity, we shall omit $\otimes \mathbb Q$ from our notation, with the understanding that the abelian groups below are to be viewed as tensored with $\mathbb Q$.  Also, let us write 
\[  \rho (\gamma_i)  \coloneqq  \rho (\mathscr S(p_{\gamma_i}))  \in  K_{1}(C_{L,0}^\ast(\widetilde M)^\Gamma). \]
To prove the theorem, it suffices to show that for any collection of elements $\{\gamma_1, \cdots, \gamma_\ell\}$ with distinct finite orders, the elements
\[  \rho(\gamma_1), \cdots, \rho(\gamma_\ell)\]  
are linearly independent in $K_{1}(C_{L,0}^\ast(\widetilde M)^\Gamma)\big/\mathcal I_1(C_{L,0}^\ast(\widetilde M)^\Gamma)$.

Let us assume the contrary, that is, there exist  $[x_1], \cdots, , [x_m] \in K_{1}(C_{L,0}^\ast(\widetilde M)^\Gamma) $ and $\psi_1, \cdots, \psi_m\in Aut_{h}(M)$ such that
\begin{equation}\label{eq:van2}
\sum_{i=1}^\ell c_i \rho(\gamma_i) = \sum_{j=1}^{m} \big([x_j] - (\widetilde \psi_{j})_\ast[x_j]\big),
\end{equation}
where $c_1, \cdots, c_\ell\in \mathbb Z$ with at least one $c_i\neq 0$. In fact, we shall study Equation $\eqref{eq:van2}$ in the group $K_1(C_{L,0}^\ast(E(\Gamma\rtimes F_m))^{\Gamma\rtimes F_m})$ and arrive at a contradiction, where $\Gamma\rtimes F_m$ is a certain semi-direct product of $\Gamma$ with the free group of $m$ generators $F_m$ and $E(\Gamma\rtimes F_m)$ is the universal space for free and proper $\Gamma\rtimes F_m$-actions. 

Let us fix a map $\sigma\colon M \to B\Gamma$ that induces an isomorphism of their fundamental groups, where $B\Gamma$ is the classifying space of $\Gamma$.  Suppose $\varphi \colon M \to M$ is an orientation preserving self homotopy equivalence of $M$. Then $\varphi$ induces an automorphism\footnote{Precisely speaking, $\varphi$ only defines an outer automorphism of $\Gamma$, and one needs to make a specific choice of a representative in $ \textup{Aut}(\Gamma)$. In any case, any such  choice will work for the proof.} of $\Gamma$, also denoted by  $\varphi\in \textup{Aut}(\Gamma)$. Now consider the semi-direct product $\Gamma\rtimes_{\varphi} \mathbb Z$ and its associated classifying space $B(\Gamma\rtimes_{\varphi} \mathbb Z)$. Let  $
\hat \varphi$ be the element in $\Gamma\rtimes_{\varphi} \mathbb Z$ that corresponds to the generator $ 1\in \mathbb Z$. We write   \[ \Phi\colon B(\Gamma\rtimes_{\varphi} \mathbb Z)\to B(\Gamma\rtimes_{\varphi} \mathbb Z) \] for the map induced by the automorphism $\Gamma\rtimes_\varphi \mathbb Z \to \Gamma\rtimes_\varphi \mathbb Z$ defined by $ a \to \hat \varphi a \hat \varphi^{-1}. $
Suppose $\iota\colon B\Gamma \to B(\Gamma\rtimes_{\varphi}\mathbb Z)$ is the map induced by the inclusion $\Gamma \hookrightarrow \Gamma\rtimes_{\varphi}\mathbb Z$.  Then the map  
\[ \iota \circ  \sigma\circ  \varphi \colon M \xrightarrow{\varphi} M \xrightarrow{\sigma} B\Gamma \xrightarrow{\iota} B(\Gamma\rtimes_{\varphi}\mathbb Z)\]  is homotopy equivalent to the map 
\[  \Phi\circ \iota \circ \sigma\colon M \xrightarrow{\sigma} B\Gamma \xrightarrow{\iota} B(\Gamma\rtimes_{\varphi}\mathbb Z)\xrightarrow{\Phi} B(\Gamma\rtimes_{\varphi}\mathbb Z),\] since they induce the same map on fundamental groups. Let $ \widetilde\sigma\colon \widetilde M \to E\Gamma$ be the lift of the map $\sigma\colon M \to B\Gamma$. Similarly, $\widetilde \varphi\colon \widetilde M\to \widetilde M$ is the lift of $\varphi\colon M \to M$, and $\widetilde \Phi\colon E(\Gamma\rtimes_{\varphi}\mathbb Z)\to E(\Gamma\rtimes_{\varphi}\mathbb Z)$ is the lift of $\Phi\colon B(\Gamma\rtimes_{\varphi}\mathbb Z)\to B(\Gamma\rtimes_{\varphi}\mathbb Z).$

Since $ \Phi\colon B(\Gamma\rtimes_{\varphi} \mathbb Z)\to B(\Gamma\rtimes_{\varphi} \mathbb Z)$ is induced by an inner conjugation morphism on $\Gamma\rtimes_{\varphi} \mathbb Z$, the map\footnote{The $C^\ast$-algebra $C_{L,0}^\ast(E\Gamma)^\Gamma$ is the  inductive limit of $C_{L,0}^\ast(Y)^\Gamma$,  where $Y$ ranges over all $\Gamma$-cocompact subspaces of $E\Gamma$. } $\widetilde \Phi_{\ast}\colon K_1(C_{L,0}^\ast(E\Gamma)^\Gamma) \to K_1(C_{L,0}^\ast(E\Gamma)^\Gamma)$ is the identity map.   It follows that for each $[x]\in K_1(C_{L,0}^\ast(\widetilde M)^\Gamma)$, we have 
\begin{align*}
   \widetilde \iota_\ast \widetilde \sigma_\ast ({\widetilde\varphi}_\ast[x])  
 = {\widetilde \Phi}_{\ast} \widetilde \iota_\ast \widetilde \sigma_\ast ([x])  =  \widetilde \iota_\ast \widetilde \sigma_\ast ([x]) 
\end{align*}   
in $ K_1(C_{L,0}^\ast(E(\Gamma\rtimes_{\varphi}\mathbb Z))^{\Gamma\rtimes_{\varphi}\mathbb Z})$, where $\widetilde \iota_\ast \widetilde \sigma_\ast$ is the composition 
\[  K_1(C_{L,0}^\ast(\widetilde M)^\Gamma) \xrightarrow{\widetilde \sigma_\ast} K_1(C_{L,0}^\ast(E\Gamma)^\Gamma) \xrightarrow{\widetilde \iota_\ast} K_1(C_{L,0}^\ast(E(\Gamma\rtimes_{\varphi}\mathbb Z))^{\Gamma\rtimes_{\varphi}\mathbb Z}). \]	   The same argument also works for an arbitrary finite number of orientation preserving self homotopy equivalences $\psi_1, \cdots, \psi_m \in Aut_{h}(M) $ simultaneously, in which case we have 
\[    \widetilde \iota_\ast \widetilde \sigma_\ast (({\widetilde\psi}_i)_\ast[x])  
  =  \widetilde \iota_\ast \widetilde \sigma_\ast ([x]) \textup{ in } K_1(C_{L,0}^\ast(E(\Gamma\rtimes_{\{\psi_1, \cdots, \psi_m\}} F_m))^{\Gamma\rtimes_{\{\psi_1, \cdots, \psi_m\}} F_m}), \]
for all $[x]\in K_1(C_{L,0}^\ast(\widetilde M)^\Gamma)$.  In other words, $({\widetilde\psi}_i)_\ast[x] $ and $[x]$ have the same image in $K_1(C_{L,0}^\ast(E(\Gamma\rtimes_{\{\psi_1, \cdots, \psi_m\}} F_m))^{\Gamma\rtimes_{\{\psi_1, \cdots, \psi_m\}} F_m})$. For notational simplicity, let us write $\Gamma\rtimes F_m$ for $\Gamma\rtimes_{\{\psi_1, \cdots, \psi_m\}} F_m$. If no confusion is likely to arise, we shall still write $[x]$ for its image $\widetilde \iota_\ast \widetilde \sigma_\ast ([x]) $ in $K_1(C_{L,0}^\ast(E(\Gamma\rtimes F_m))^{\Gamma\rtimes F_m})$.

		If we pass Equation $\eqref{eq:van2}$ to $K_1(C_{L,0}^\ast(E(\Gamma\rtimes F_m))^{\Gamma\rtimes F_m})$ under the map \[  K_1(C_{L,0}^\ast(\widetilde M)^\Gamma) \xrightarrow{\widetilde \sigma_\ast} K_1(C_{L,0}^\ast(E\Gamma)^\Gamma) \xrightarrow{\widetilde \iota_\ast} K_1(C_{L,0}^\ast(E(\Gamma\rtimes F_m))^{\Gamma\rtimes F_m}), \] then it follows from the above discussion that 
		\[  \sum_{k=1}^\ell c_k \rho(\gamma_k) = 0 \textup{ in $K_1(C_{L,0}^\ast(E(\Gamma\rtimes F_m))^{\Gamma\rtimes F_m})$}  , \]
		where at least one $c_k\neq 0$. We have 
\begin{equation}\label{eq:vanish}
\partial_{\Gamma\rtimes F_m} \Big( \sum_{k=1}^\ell c_k [p_{\gamma_k}] \Big) = 2 \cdot \big( \sum_{k=1}^\ell c_k \rho(\gamma_k) \big) = 0,
\end{equation}
where $ \partial_{\Gamma\rtimes F_m} $ is the connecting map in the following long exact sequence: 
\begin{equation}
\begin{gathered}
\resizebox{14cm}{!}{\xymatrix{K_0(C_{L,0}^\ast(E(\Gamma\rtimes F_m))^{\Gamma\rtimes F_m})\ar[r] &  K_0^{\Gamma\rtimes F_m}(E(\Gamma\rtimes F_m))  \ar[r]^{\mu} &  K_0(C^\ast(\Gamma\rtimes F_m)) \ar[d]^{\partial_{\Gamma\rtimes F_m}} \\
	K_1(C^\ast(\Gamma\rtimes F_m)) \ar[u]& K_{1}^{\Gamma\rtimes F_m}(E(\Gamma\rtimes F_m)) \ar[l] &  K_{1}(C_{L,0}^\ast(E(\Gamma\rtimes F_m))^{\Gamma\rtimes F_m}) \ar[l]} }
\end{gathered}
\end{equation}		

Now by assumption $\Gamma$ is strongly finitely embeddable into Hilbert space. Hence $\Gamma\rtimes F_m$ is finitely embeddable into Hilbert space. By Proposition $\ref{prop:rank}$, we have the following. 
		\begin{enumerate}[(i)]
			\item $\{[p_{\gamma_1}], \cdots, [p_{\gamma_\ell}]\}$ generates a rank $n$ abelian subgroup of $K_0^\fin(C^\ast(\Gamma\rtimes F_m))$, since $\gamma_1, \cdots, \gamma_n$ have distinct finite orders. In other words, 
			\[ \sum_{k=1}^n c_k [p_{\gamma_k}] \neq 0  \in K_0^\fin(C^\ast(\Gamma\rtimes F_m))\]
			if at least one $c_k\neq 0$.
			\item Every nonzero element in $K_0^\fin(C^\ast(\Gamma\rtimes F_m))$ is not in the image of the assembly map 
			\[ \mu\colon  K_0^{\Gamma\rtimes F_m}(E(\Gamma\rtimes F_m)) \to K_0(C^\ast(\Gamma\rtimes F_m)). \]
			 In particular, we see that 
		$ \partial_{\Gamma\rtimes F_m}\colon  K^\fin_0(C^\ast(\Gamma\rtimes F_m)) \to K_{1}(C_{L,0}^\ast(\widetilde X)^{\Gamma\rtimes F_m})$
			is injective.  
		\end{enumerate} 
It follows that $\partial_{\Gamma\rtimes F_m} \Big( \sum_{k=1}^\ell c_k [p_{\gamma_k}] \Big) \neq 0$, which contradicts Equation $\eqref{eq:vanish}$.
This finishes the proof.
	\end{proof}

It is tempting to use a similar argument to prove an analogue of Theorem $\ref{thm:structure}$ above for  $\widetilde{\mathcal S}_{geom}(M)$. However, there are some subtleties. 
First, note that  
	\[  \alpha_\varphi(\theta) +  [\varphi] =  \beta_\varphi (\theta)    \]	
	for all $\theta = (f, N) \in \mathcal S(M)$ and all $\varphi\in Aut_{h}(M)$, where $[\varphi] = (\varphi, M)$ is the element given by $\varphi\colon M \to M$ in $ \mathcal S(M). $
	It follows that 
	\[  \rho(\beta_\varphi(\theta))  = \rho(\alpha_\varphi(\theta)) + \rho([\varphi]) = \varphi_\ast(\rho(\theta)) + \rho([\varphi]). \]
	In other words, in general,  $\rho(\beta_\varphi (\theta)) \neq \varphi_\ast(\rho(\theta))$, and consequently the homomorphism \[ \rho\colon {\mathcal S}(M) \to  K_1(C_{L,0}^\ast(\widetilde M)^\Gamma) \]
	does \emph{not} descend to a homomorphism from $\widetilde{\mathcal S}_{geom}(M)$ to $ K_{1}(C_{L,0}^\ast(\widetilde M)^\Gamma)\big/\mathcal I_1(C_{L,0}^\ast(\widetilde M)^\Gamma)$. 
New ideas are needed to take care of this issue.  On the other hand, there is strong evidence which suggests an analogue of Theorem $\ref{thm:structure}$ for  $\widetilde{\mathcal S}_{geom}(M)$. For example, this has been verified by Weinberger and Yu for residually finite groups \cite[Theorem 3.9]{WY}. Also, Chang and Weinberger showed that the free rank of $\widetilde{\mathcal S}_{geom}(M)$ is at least $1$ when  $\pi_1 X = \Gamma$ is not torsion free \cite[Theorem 1]{ChW}. 

The above discussion motivates the following conjecture.

\begin{conj}
	Let $M$ be a closed oriented topological manifold with dimension $ n = 4k-1$ \textup{($k>1$)} and $\pi_1 M = \Gamma$. Then the free ranks of  $\widetilde{\mathcal S}_{alg}(M)$ and $\widetilde{\mathcal S}_{geom}(M)$ are $\geq N_\fin(\Gamma)$. 
\end{conj}

We conclude this section by proving the following theorem, which is an analogue of the theorem of Chang and Weinberger cited above \cite[Theorem 1]{ChW}.

\begin{thm}
	Let $X$ be a closed oriented topological manifold with dimension $ n = 4k-1$ \textup{($k>1$)} and $\pi_1 X = \Gamma$. If $\Gamma$ is not torsion free, then the free rank of $\widetilde{\mathcal S}_{alg}(X)$  is $\geq 1$. 
\end{thm}
\begin{proof}
	Recall that for any non-torsion-free countable discrete group $G$, if $\gamma \neq e$ is a finite order element of $G$, then $[p_\gamma]$ generates a subgroup of rank one in $K_0(C^\ast(G))$ and any nonzero multiple of $[p_\gamma]$ is not in the image of the assembly map $\mu\colon K_0^\Gamma(EG)\to K_0(C^\ast(G))$ \cite{WY}. Using this fact, the statement follows from the same proof as in Theorem $\ref{thm:structure}$.  
\end{proof}

\section{Cyclic cohomology and  higher rho invariants}

Connes' cyclic cohomology theory provides a powerful method to compute higher rho invariants.
In this section, we give a survey of recent work on the pairing between Connes' cyclic cohomology and $C^\ast$-algebraic secondary invariants. In the case of higher rho invariants given by invertible\footnote{Here ``invertible'' means being invertible on the universal cover of the manifold.} operators on manifolds, this pairing can be computed in terms of Lott's higher eta invariants. We apply these results to the higher Atiyah-Patodi-Singer index theory and discuss a potential way to construct counter examples to the Baum-Connes conjecture.

We shall first discuss the zero dimensional cyclic cocycle case. 
Let $M$ be a spin Riemannian manifold with positive scalar curvature and let $D$ be the Dirac operator on $M$. Let  $\widetilde M $ be the universal cover of $M$ and $\widetilde D$ the lifting of $D$.
Lott introduced the following delocalized eta invariant $\eta_{\langle h\rangle}(\widetilde D)$ \cite{Lo1}: 
\begin{equation}\label{eq:delocaleta}
\eta_{\langle h\rangle}(\widetilde D) \coloneqq  \frac{2}{\sqrt \pi} \int_{0}^{\infty}  \tr_{\langle h\rangle}(\widetilde D e^{-t^2 \widetilde D^2})dt,  
\end{equation} 
 under the condition that the conjugacy class $\langle h\rangle$ of $h\in \Gamma = \pi_1M$ has polynomial growth. Here  $\Gamma = \pi_1 M$ is the fundamental group of $M$, and   the trace map $\tr_{\langle h\rangle}$ is defined as follows:   \[ \tr_{\langle h\rangle}(A) = \sum_{g\in \langle h \rangle} \int_{\mathcal F} A(x, gx) dx   \] 
on $\Gamma$-equivariant Schwartz kernels $A\in C^\infty(\widetilde M\times \widetilde M)$, where $\mathcal F$ is a fundamental domain of $\widetilde M$ under the action of $\Gamma$. 
  
We have the following theorem \cite{XY3}. 
\begin{thm}\label{thm:intro}
	Let $M$ be a closed odd-dimensional spin manifold equipped with a positive scalar curvature metric $g$. Suppose $\widetilde M$ is the universal cover of $M$, $\tilde g$ is the Riemannnian metric on $\widetilde M$  lifted from $g$, and $\widetilde D$ is the associated Dirac operator on $\widetilde M$. Suppose the conjugacy class $\langle h\rangle$ of a non-identity element $h\in \pi_1M$ has polynomial growth,   then we have
	\[ \tau_h (\rho(\widetilde D, \widetilde g) ) = - \frac{1}{2}\eta_{\langle h\rangle}(\widetilde D),   \]
	where $\rho(\widetilde D, \widetilde g)$ is the $K$-theoretic  higher rho invariant of $\widetilde D$ with respect to the metric $\tilde g$, and $\tau_h$ is a canonical determinant map associated to $\langle h \rangle$. 
\end{thm}

As an application of Theorem $\ref{thm:intro}$ above, we have the following algebraicity result concerning the values of delocalized eta invariants \cite{XY3}.

\begin{thm}
	With the same notation as above, if the rational Baum-Connes conjecture holds for $\Gamma$, and the conjugacy class $\langle h\rangle$ of a non-identity element $h\in \Gamma$ has polynomial growth,  then
	the delocalized eta invariant $\eta_{\langle h\rangle}(\widetilde D)$ is an algebraic number.  Moreover,  if in addition $h$ has infinite order, then  $\eta_{\langle h\rangle}(\widetilde D)$ vanishes.   
\end{thm}

This theorem follows from the construction of the determinant map $\tau_h$ and a $L^2$-Lefschetz fixed point theorem of B.-L. Wang and H. Wang  \cite[Theorem 5.10]{WW}.  When $\Gamma$ is torsion-free and satisfies the Baum-Connes conjecture,  and the conjugacy class $\langle h\rangle$ of a non-identity element $h\in \Gamma$ has polynomial growth, Piazza and Schick have proved the vanishing of  $\eta_{\langle h\rangle}(\widetilde D)$ by a different method \cite[Theorem 13.7]{PS}. 

In light of this algebraicity result, we propose the following open question.

\begin{open} If the conjugacy class $\langle h\rangle$ of a non-identity element $h\in \Gamma$ has polynomial growth,
	what values can  the delocalized eta invariant $\eta_{\langle h\rangle}(\widetilde D)$ take in general? Are they always algebraic numbers?
\end{open}

In particular, if a delocalized eta invariant is transcendental, then it will lead to a counterexample to the Baum-Connes conjecture \cite{BC, BCH, C}. Note that the above question is a reminiscent of Atiyah's question concerning rationality of $\ell^2$-Betti numbers \cite{A1}. Atiyah's question was answered in negative by Austin, who showed that $\ell^2$-Betti numbers can be transcendental \cite{Au}. 

So far, we have been assuming the conjugacy class $\langle h \rangle$ has polynomial growth, which guarantees the convergence of the integral in $\eqref{eq:delocaleta}$. In general, the integral in $\eqref{eq:delocaleta}$ fails to converge. The following theorem of Chen-Wang-Xie-Yu  \cite{CWXY} gives a sufficient condition for when the integral in  $\eqref{eq:delocaleta}$ converges. 

\begin{thm}\label{thm:converge}
	Let $M$ be a closed manifold and $\widetilde M$ the universal covering over $M$. Suppose $D$ is a self-adjoint first-order elliptic differential operator over $M$ and $\widetilde D$ the lift of $D$ to $\widetilde M$. If $\left\langle h\right\rangle$ is a nontrivial conjugacy class of $\pi_1(M)$ and $\widetilde D$ has a sufficiently large spectral gap at zero, then the delocalized eta invariant $\eta_{\left\langle h\right\rangle }(\widetilde D)$ of $\widetilde D$ is well-defined.
\end{thm}

We would like to emphasis that the theorem above works for all fundamental groups.  In the special case where the conjugacy class $\langle h\rangle$ has sub-exponential growth, then any nonzero spectral gap is in fact sufficiently large, hence in this case $\eta_{\left\langle h\right\rangle }(\widetilde D)$ is well-defined as long as $\widetilde D$ is invertible.  

Let us make precise of what ``sufficiently large spectral gap at zero" means.  Fix a finite generating set $S$ of $\Gamma$. Let  $\ell$ be the corresponding word length function on $\Gamma$ determined by $S$. Since $S$ is finite, there exist $C$ and $K_{\langle h\rangle}>0$ such that 
\begin{equation}\label{eq:growthofh}
\#\{g\in\langle h\rangle:\ell(g)=n\}\leqslant C e^{K_{\langle h\rangle}\cdot n}.
\end{equation}
We define $\uptau_{\langle h\rangle}$ to be 
\begin{equation}\label{eq:equivalentwithlength}
\uptau_{\langle h\rangle}=\liminf_{\substack{g\in\langle h\rangle \\ \ell(g)\to\infty}} \Big(\inf_{x\in\widetilde M}\frac{\dist(x,gx)}{\ell(g)}\Big).
\end{equation}
Since  the action of $\Gamma$ on $\widetilde M$ is free and cocompact, we have $\uptau_{\langle h\rangle} > 0$.

We denote the principal symbol of $D$ by $\sigma_D(x, v)$, for $x\in M$ and cotangent vector $v\in T_x^\ast M$. 
We define the propagation speed of $D$ to be the  positive number 
\begin{equation*}
c_D = \sup\{ \|\sigma_D(x, v)\| \colon x\in M, v\in T^\ast_xM, \|v\| =1\}. 
\end{equation*} 
\begin{defn}
	With the above notation, let us define	\begin{equation}\label{eq:sigma_h}
	\sigma_{\langle h\rangle}\coloneqq \frac{2K_{\langle h\rangle}\cdot c_D}{\uptau_{\langle h\rangle}}.
	\end{equation}
\end{defn}

Recall that  $\widetilde D$ is said to have a spectral gap at zero if there exists an open interval $(-\varepsilon, \varepsilon)\subset \mathbb R$ such that $\textup{spectrum}(\widetilde D) \cap (-\varepsilon, \varepsilon)$ is either $\{0\}$ or empty.  Moreover, $\widetilde D$ is said to have a \emph{sufficiently large} spectral gap at zero if its spectral gap is larger than $\sigma_{\langle h\rangle}$.

Again it is  natural to ask the following question.

\begin{open} With $\widetilde D$ as in the above theorem,
	what values can  the delocalized eta invariant $\eta_{\langle h\rangle}(\widetilde D)$ take in general? Are they always algebraic numbers?
\end{open}

A special feature of traces is that they always have uniformly bounded representatives, when viewed as degree zero cyclic cocycles. In fact, our proof of Theorem $\ref{thm:converge}$  allows us to generalize Theorem $\ref{thm:converge}$ to cyclic cocycles of higher degrees, as long as they have at most exponential growth.  Recall that the cyclic cohomology of a group algebra $\mathbb C\Gamma$ has a decomposition respect to the conjugacy classes of $\Gamma$ (\cite{Nis}):
\[ HC^*(\mathbb C\Gamma)\cong\prod_{\left\langle h\right\rangle }HC^*(\mathbb C\Gamma,\left\langle h\right\rangle), \]
where 
$HC^*(\mathbb C\Gamma,\left\langle h\right\rangle)$  denotes the component that corresponds to the conjugacy class $\langle h\rangle$. If $\langle h\rangle$ is a nontrivial conjugacy class, then a cyclic cocycle in $HC^*(\mathbb C\Gamma,\left\langle h\right\rangle)$  will be called a delocalized cyclic cocycle at  $\langle h\rangle$.

\begin{thm}\label{thm:higherconverge}
	Assume the same notation as in Theorem $\ref{thm:converge}$. Let $\varphi$ be a delocalized cyclic cocycle at a nontrivial conjugacy class $\langle h\rangle$. If $\varphi$ has exponential growth and $\widetilde D$ has a sufficiently large spectral gap at zero, then   $\eta_{\varphi}(\widetilde D)$ is well-defined, where  $\eta_{\varphi}(\widetilde D)$ is a higher analogue \textup{(}cf. \cite{CWXY}\textup{)} of the formula $\eqref{eq:delocaleta}$ .
\end{thm}

  For higher degree cyclic cocycles, the precise meaning of ``sufficiently large spectral gap at zero'' is similar to but slightly different from that of the case of traces. We refer the reader to \cite[Section 3.2]{CWXY} for more details. 
  For now, we simply point out that if both $\Gamma$ and $\varphi$ have sub-exponential growth,  then any nonzero spectral gap is in fact sufficiently large, hence in this case $\eta_{\varphi}(\widetilde D)$ is well-defined as long as $\widetilde D$ is invertible.   The explicit formula for  $\eta_{\varphi}(\widetilde D)$ is described in terms of  the transgression formula for the Connes-Chern character \cite{C, C2}. 
  It is essentially\footnote{We refer the reader to \cite{CWXY} for details on how to identify the formula for $\eta_{\varphi}(\widetilde D)$ in Theorem $\ref{thm:higherconverge}$ with  the periodic version of Lott's noncommutative-differential higher eta invariant.} a periodic version of the delocalized part of Lott's noncommutative-differential higher eta invariant.  We shall call $\eta_{\varphi}(\widetilde D)$ a \emph{delocalized higher eta invariant} from now on.

    Formally speaking, just as Lott's delocalized eta invariant $\eta_{\left\langle h\right\rangle }(\widetilde D)$ can be interpreted as the pairing between the degree zero cyclic cocycle $\tr_{\langle h\rangle}$ and the higher rho invariant $\rho(\widetilde D)$, so can the delocalized higher eta invariant $\eta_{\varphi}(\widetilde D)$ be interpreted as the pairing between the cyclic cocycle $\varphi$ and the higher rho invariant $\rho(\widetilde D)$. A key analytic difficulty here is to verify when such a pairing is well-defined, 
    or more ambitiously, to verify when one can extend this pairing to a pairing between the cyclic cohomology  of $\mathbb C\Gamma$ and the $K$-theory group $K_\ast(C_{L,0}^\ast(\widetilde M)^\Gamma)$. The group $K_\ast(C_{L,0}^\ast(\widetilde M)^\Gamma)$ consists of $C^\ast$-algebraic secondary  invariants; in particular, it contains all higher rho invariants from the discussion above. Such an extension of the pairing is important, often necessary, for many interesting applications to geometry and topology (cf. \cite{PS1, XY1, WXY}).

In \cite{CWXY},  such an extension of the pairing, that is, a pairing between delocalized cyclic cocycles of \emph{all degrees} and the $K$-theory group $K_\ast(C_{L,0}^\ast(\widetilde M)^\Gamma)$ was established, in the case of Gromov's hyperbolic groups. More precisely, we have the following theorem  \cite{CWXY}.
 \begin{thm}\label{main}
 	Let $M$ be a closed manifold whose fundamental group $\Gamma$ is  hyperbolic. Suppose  $\left\langle h\right\rangle $  is non-trivial conjugacy class of $\Gamma$. Then every element $[\alpha]\in HC^{2k+1-i}(\mathbb C\Gamma,\left\langle h\right\rangle ) $ induces a natural map 
 	$$\tau_{[\alpha]} \colon K_i(C^*_{L,0}(\widetilde M)^\Gamma)\to \mathbb C$$ such that the following are satisfied: 
 	\begin{enumerate}[$(i)$]
 		\item $\tau_{[S\alpha]} = \tau_{[\alpha]}$, where $S$ is Connes' periodicity map \[ S\colon HC^{\ast}(\mathbb C\Gamma,\left\langle h\right\rangle )\to HC^{\ast+2}(\mathbb C\Gamma,\left\langle h\right\rangle );\]
 		\item if $D$ is an elliptic operator on $M$ such that the lift $\widetilde D$ of $D$ to the universal cover $\widetilde M$ of $M$ is invertible, then we have 
 		$$\tau_{[\alpha]}(\rho(\widetilde D))=\eta_{[\alpha]}(\widetilde D),  $$
 		where $\rho(\widetilde D)$ is the higher rho invariant of $\widetilde D$ and $\eta_{[\alpha]}(\widetilde D)$ is the  delocalized higher eta invariant from Theorem $\ref{thm:higherconverge}$. In particular, in the case of hyperbolic groups, the delocalized higher eta invariant $\eta_{[\alpha]}(\widetilde D)$  is always well-defined, as long as $\widetilde D$ is invertible. 
 	\end{enumerate}
 \end{thm}

The construction of the map $\tau_{[\alpha]}$ in the above theorem uses Puschnigg's smooth dense subalgebra  for hyperbolic groups  \cite{P1} in an essential way. In more conceptual terms, the above theorem provides an explicit formula to compute the  delocalized Connes-Chern character of $C^\ast$-algebraic secondary invariants. More precisely, the same techniques developed in \cite{CWXY} actually imply\footnote{In fact, even more is true. One can use the same techniques developed in \cite{CWXY} to show that if $\mathcal A$ is   smooth dense subalgebra of $C_r^\ast(\Gamma)$ for any group $\Gamma$ (not necessarily hyperbolic) and in addition $\mathcal A$ is a Fr\'echet locally $m$-convex algebra, then there is a well-defined delocalized Connes-Chern character $Ch_{deloc}\colon K_i(C^*_{L,0}(\widetilde M)^\Gamma)\to \overline{HC}^{deloc}_{\ast}(\mathcal A)$. Of course,  in order to pair such a delocalized Connes-Chern character with a cyclic cocycle of $\mathbb C\Gamma$, the key remaining challenge is to  continuously extend this cyclic cocycle of $\mathbb C\Gamma$ to a cyclic cocycle of $\mathcal A$.  } that there is a well-defined delocalized Connes-Chern character $Ch_{deloc}\colon K_i(C^*_{L,0}(\widetilde M)^\Gamma)\to \overline{HC}^{deloc}_{\ast}(\mathcal B)$, where $\mathcal B$ is Puschnigg's smooth dense subalgebra of $C_r^\ast(\Gamma)$ and $\overline{HC}^{deloc}_{\ast}(\mathcal B)$ is the delocalized part of the cyclic homology\footnote{Here the definition of cyclic homology of $\mathcal B$ takes the topology of $\mathcal B$ into account, cf. \cite[Section II.5]{C2}. } of $\mathcal B$. Now for Gromov's hyperbolic groups, every cyclic cohomology class of $\mathbb C\Gamma$ continuously extends to  cyclic cohomology class of $\mathcal B$ (cf. \cite{P1} for the case of degree zero cyclic cocycles and \cite{CWXY} for the case of higher degree cyclic cocycles). Thus  the map $\tau_{[\alpha]}$ can be viewed as a pairing between  cyclic cohomology and delocalized Connes-Chern characters of $C^\ast$-algebraic secondary invariants.  As a consequence, this unifies Higson-Roe's higher rho invariant and Lott's higher eta invariant for invertible operators.

We point out that the proof of Theorem $\ref{main}$ does \emph{not} rely on the Baum-Connes isomorphism for hyperbolic groups \cite{L, MY}, although the theorem is closely connected to the Baum-Connes conjecture and the Novikov conjecture. On the other hand, if one is willing to use the full power of the Baum-Connes isomorphism for hyperbolic groups, there is in fact a different, but more indirect, approach to the delocalized Connes-Chern character map. First, observe that the map $\tau_{[\alpha]}$ factors through a map
\[ \tau_{[\alpha]}\colon  K_i(C^*_{L,0}(\underline{E}
\Gamma)^\Gamma)\otimes \mathbb C\to \mathbb C\]
where $\underline{E}\Gamma$ is the universal space for proper $\Gamma$-actions.  Now the Baum-Connes isomorphism $\mu\colon K_\ast^G(\underline{E}\Gamma) \xrightarrow{ \ \cong \ }K_\ast(C^\ast_r(\Gamma))$ for hyperbolic groups implies that one can identify $K_i(C^*_{L,0}(\underline{E}\Gamma)^\Gamma)\otimes \mathbb C$  with  $\bigoplus_{\left\langle h\right\rangle \neq 1}HC_\ast(\mathbb C\Gamma,\left\langle h\right\rangle)$, where $HC_\ast(\mathbb C\Gamma,\left\langle h\right\rangle)$ is the delocalized cyclic homology at $\langle h\rangle$ and the direct sum is taken over all nontrivial conjugacy classes. In particular, after this identification, it follows that the map $\tau_{[\alpha]}$ becomes the usual pairing between cyclic cohomology and cyclic homology. However, for a specific element, e.g. the higher rho invariant $\rho(\widetilde D)$,  in  $K_i(C^*_{L,0}(\underline{E}\Gamma)^\Gamma)$, its identification with an element in  $\bigoplus_{\left\langle h\right\rangle \neq 1}HC_\ast(\mathbb C\Gamma,\left\langle h\right\rangle)$ is rather abstract and implicit. More precisely,  the   computation of the number  $\tau_{[\alpha]}(\rho(\widetilde D))$ essentially amounts to the following process. Observe that if a closed spin manifold $M$ is equipped with a positive scalar curvature metric, then stably it bounds (more precisely,  the universal cover $\widetilde M$ of $M$   becomes the boundary of another $\Gamma$-manifold,  after finitely many steps of cobordisms and  vector bundle modifications). In principle,  the number $\tau_{[\alpha]}(\rho(\widetilde D))$ can be derived from a higher Atiyah-Patodi-Singer index theorem for this bounding manifold. Again, there is a serious drawback of such an indirect approach --- the explicit  formula for $\tau_{[\alpha]}(\rho(\widetilde D))$ is completely lost. In contrast, a key feature of the construction of the delocalized Connes-Chern character map in Theorem $\ref{main}$ is that the formula is explicit and  intrinsic.

In \cite{DG}, Deeley and Goffeng also constructed an implicit delocalized Chern character map for \mbox{$C^\ast$-algebraic} secondary invariants. Their approach is in spirit  similar to the indirect method just described above (making use of the Baum-Connes isomorphism for hyperbolic groups),   although their actual technical implementation is different.

 As an application, we use this delocalized Connes-Chern character map from Theorem $\ref{main}$ to  derive a delocalized higher Atiyah-Patodi-Singer index theorem for manifolds with boundary. More precisely, let  $W$ be a compact $n$-dimensional spin manifold with boundary $\partial W$. Suppose $W$ is equipped with a Riemannian metric $g_W$ which has product structure near $\partial W$ and in addition has positive scalar curvature on $\partial W$. Let $\widetilde W$ be the universal covering of $W$ and $g_{\widetilde W}$  the   Riemannian metric  on $\widetilde W$ lifted from $g_W$.  With respect to the metric $g_{\widetilde W}$,  the associated Dirac operator $\widetilde D_W$ on $\widetilde  W$ naturally defines a higher index $Ind_\Gamma(\widetilde D_W)$ (as in Section $\ref{sec:rhobdry}$) in $K_n(C^\ast(\widetilde W)^\Gamma) = K_n(C_r^\ast(\Gamma))$, where $\Gamma = \pi_1(W)$. Since the metric $g_{\widetilde W}$ has positive scalar curvature on $\partial \widetilde W$, it follows from the Lichnerowicz formula that the associated Dirac operator $\widetilde D_\partial$ on $\partial {\widetilde W}$ is invertible, hence naturally defines a higher rho invariant $\rho(\widetilde D_\partial)$ in $K_{n-1}(C_{L, 0}^\ast(\widetilde W)^\Gamma)$. We have the following delocalized higher Atiyah-Patodi-Singer index theorem. 

\begin{thm}\label{thm:delocalaps}
	 With the notation as above, if $\Gamma=\pi_1(W)$ is hyperbolic and $\left\langle h\right\rangle $ is a nontrivial conjugacy class of $\Gamma$, then for any $[\varphi]\in HC^{\ast}(\mathbb C\Gamma,\left\langle h\right\rangle )$,
	we have \begin{equation}
	Ch_{[\varphi]}(Ind_\Gamma(\widetilde D_W))=-\frac{1}{2}\eta_{[\varphi]}(\widetilde D_\partial ), 
	\end{equation}
	where $Ch_{[\varphi]}(Ind_\Gamma(\widetilde D_W))$ is the Connes-Chern pairing between the cyclic cohomology class $[\varphi]$ and the higher index class $Ind_\Gamma(\widetilde D_W)$.
\end{thm}
\begin{proof}
This follows from Theorem $\ref{main}$ and Theorem $\ref{thm:bdry}$.
\end{proof}
 
By using Theorem $\ref{main}$, we have derived Theorem $\ref{thm:delocalaps}$ as a consequence of a  $K$-theoretic counterpart. This  is possible only because we have realized  $\eta_{[\varphi]}(\widetilde D_\partial)$ as the pairing between the cyclic cocycle $\varphi$ and the $C^\ast$-algebraic secondary invariant $\rho(\widetilde D_\partial)$ in $K_1(C_{L,0}^\ast(\widetilde W)^\Gamma)$. 


Alternatively, one can also derive Theorem  $\ref{thm:delocalaps}$ from a version of  higher Atiyah-Patodi-Singer index theorem due to Leichtnam and Piazza \cite[Theorem 4.1]{LP3} and Wahl \cite[Theorem 9.4 \& 11.1]{Wa}. This version of  higher Atiyah-Patodi-Singer index theorem is stated in terms of noncommutative differential forms on a smooth dense subalgebra of $C_r^\ast(\Gamma)$; or noncommutative differential forms on a certain class of smooth dense subalgebras (if exist) of general $C^\ast$-algebras (not just group $C^\ast$-algebras) in Wahl's version.  In the case of Gromov's hyperbolic groups, one can choose such a smooth dense subalgebra to be Puschnigg's smooth dense subalgebra $\mathcal B$. As mentioned before, for Gromov's hyperbolic groups, every cyclic cohomology class of $\mathbb C\Gamma$ continuously extends to a  cyclic cohomology class of $\mathcal B$ (cf. \cite{P1} for the case of degree zero cyclic cocycles and \cite{CWXY} for the case of higher degree cyclic cocycles). Now Theorem $\ref{thm:delocalaps}$ follows by pairing the higher Atiyah-Patodi-Singer index formula of Leichtnam-Piazza and Wahl with the delocalized cyclic cocycles of $\mathbb C\Gamma$. 

One can also try to pair the higher Atiyah-Patodi-Singer index formula of Leichtnam-Piazza and Wahl with group cocycles of $\Gamma$, or equivalently cyclic cocycles in  $HC^{\ast}(\mathbb C\Gamma,\left\langle 1\right\rangle )$, where $\langle 1\rangle$ stands for the conjugacy class of the identity element of $\Gamma$. In this case, for fundamental groups with property RD,  Gorokhovsky, Moriyoshi and Piazza proved a higher Atiyah-Patodi-Singer index theorem for group cocycles with polynomial growth \cite{Gr}.

\bigskip
\footnotesize

\end{document}